\documentclass[reqno]{amsart}

\usepackage{amsthm}
\usepackage[usenames,dvipsnames]{pstricks}
\usepackage{epsfig}
\usepackage{pst-grad} 
\usepackage{pst-plot} 
\usepackage{graphicx,subfigure}
\usepackage{amsmath,amssymb}

\usepackage{acronym}
\acrodef{BO}{{\sl Benjamin-Ono}}
\acrodef{rBO}{{\sl regularized Benjamin-Ono}}
\acrodef{rILW}{{\sl regularized Intermediate Long Wave}}
\acrodef{DSW}{{\sl Dispersive Shock Wave}}
\acrodef{DSWs}{{\sl Dispersive Shock Waves}}
\acrodef{ILW}{{\sl Intermediate Long Wave}}
\acrodef{CGN}{{\sl Conjugate Gradient-Newton}}
\acrodef{SW/SW}{{\sl Shallow water / Shallow water}}
\acrodef{B/B}{{\sl Boussinesq / Boussinesq}}

\makeatletter
\newcommand{\sech}{\mathop{\operator@font sech}}
\newcommand{\sign}{\mathop{\operator@font sign}}
\makeatother

\newtheorem{lemma}{Lemma}[section]
\newtheorem{theorem}{Theorem}[section]
\newtheorem{proposition}{Proposition}[section]

\newtheorem{remark}{Remark}[section]

\numberwithin{equation}{section}

\begin{document}

\title[]{A high-order method for the numerical approximation of fractional nonlinear Schr\"{o}dinger equations}


\author{Angel Dur\'an}
\address{\textbf{A.~Dur\'an:} Applied Mathematics Department, University of Valladolid, P/ Belen 15, 47011, Valladolid, Spain}
\email{angeldm@uva.es}

\author{Nuria Reguera}
\address{\textbf{N.~Reguera:} Department of Mathematics and Computation, University of Burgos, 09001 Burgos, Spain}
\email{nreguera@ubu.es}







\begin{abstract}
In this paper, the periodic initial-value problem for the fractional nonlinear Schr\"{o}dinger (fNLS) equation is discretized in space by a Fourier spectral Galerkin method and in time by diagonally implicit, high-order Runge-Kutta schemes, based on the composition with the implicit midpoint rule (IMR). Some properties and error estimates for the semidiscretization in space and for the full discretization are proved. The convergence results and the general performance of the scheme are illustrated with several numerical experiments.
\end{abstract}

\maketitle

\section{Introduction}\label{sec1}
In this paper, we will consider the cubic fNLS equation
\begin{eqnarray}
iu_{t}- (-\partial_{xx})^{s}u+ f(u)=0,\quad {x}\in \mathbb{R},\quad t>0,\label{fnls11}
\end{eqnarray} 
where $0<s\leq 1$, the Laplacian-type operator $(-\partial_{xx})^{s}$ s has the Fourier representation
\begin{eqnarray}
\widehat{(-\partial_{xx})^{s}u}(\xi)=|\xi|^{2s}\widehat{u}(\xi),\quad \xi\in\mathbb{R},\label{fnls11b}
\end{eqnarray}
(where $\widehat{u}(\xi)$ denotes the Fourier transform of $u$ at $\xi$), and $f:\mathbb{C}\rightarrow\mathbb{C}$ is the cubic term $f(u)=|u|^{2}u$. Note that, if $\overline{z}$ denotes the complex conjugate of $z$, then
\begin{eqnarray}
f(z)=\frac{\partial}{\partial \overline{z}}V(z),\;  z\in\mathbb{C},\; V(z)=\frac{|z|^{4}}{2}.\label{fnls11a}
\end{eqnarray}
Equation (\ref{fnls11}) was originally introduced by Laskin, \cite{Laskin2000,Laskin2002,Laskin2011} in his mathematical principles of fractional quantum mechanics, see also \cite{FrohlichJL2007}. Additional areas of application are nonlinear optics, \cite{Malomed1,Malomed2} and references therein, and water wave models, \cite{IonescuP2014,ObrechtS2015}
The alternative formulation of (\ref{fnls11}) as a real system has the form
\begin{eqnarray}
v_{t}-(-\partial_{xx})^{s}w+{\rm Im}f(v,w)&=&0,\nonumber\\
-w_{t}-(-\partial_{xx})^{s}v+{\rm Re}f(v,w)&=&0,\label{fnls12}
\end{eqnarray}
for $u=v+iw$, $v, w$ real-valued functions, and where $f$ is viewed as a vector field $f:\mathbb{R}^{2}\rightarrow\mathbb{R}^{2}$.

\subsection{Properties of the periodic initial-value problem and numerical approximation}
The aim of the present paper is the numerical analysis of the periodic initial-value problem (ivp) of (\ref{fnls11}) on some interval $(-L,L)$ with
\begin{eqnarray}
u(x,0)=u_ {0}(x),\; x\in (-L,L),\label{fnls13a}
\end{eqnarray}
for some $2L$-periodic initial condition $u_{0}$ or, alternatively, the periodic ivp for (\ref{fnls12}) with $2L$-periodic initial data
\begin{eqnarray}
v(x,0)=v_{0}(x),\; w(x,0)=w_{0}(x),\; x\in (-L,L).\label{fnls12a}
\end{eqnarray}
In addition, (\ref{fnls11b}) holds for $\xi=k\in\mathbb{Z}$, where $\widehat{u}(k)$ denotes the $k$th Fourier component of $u$ in $(-L,L)$, that is
\begin{eqnarray*}
\widehat{u}(k)=\frac{1}{2L}\int_{-L}^{L}e^{-i\frac{\pi kx}{L}}u(x)dx.
\end{eqnarray*}
The numerical analysis of the problem will be made in sections \ref{sec2} and \ref{sec3}, for simplicity, taking $L=\pi$.

The following notation will be used throughout the paper. $L^{2}=L^{2}(-L,L)$ will denote the space of squared integrable functions on $(-L,L)$, with inner product
\begin{eqnarray}
(u_{1},u_{2})=\int_{-L}^{L}u_{1}(x)\overline{u_{2}(x)}dx,\label{fnlsL3}
\end{eqnarray}
with induced norm $||\cdot||$. When $u\in L^{2}$ is represented by its real and imaginary parts $u=(v,w)\in L^{2}\times L^{2}$, the real part of (\ref{fnlsL3}) will also be used as inner product, and will be denoted by $\langle\cdot,\cdot\rangle$. The norm in $L^{\infty}(-L,L)$ is denoted by $|\cdot|_{\infty}$. For real $\mu\geq 0$, $H^{\mu}$ will stand for the $L^{2}$-based Sobolev space of order $\mu$ of periodic functions on $(-L,L)$ with associated norm
\begin{eqnarray*}
||u||_{\mu}=\left(\sum_{k\in\mathbb{Z}}(1+k^{2})^{\mu}|\widehat{u}(k)|^{2}\right)^{1/2}.
\end{eqnarray*}
An expression of the form $g_{1} \lesssim g_{2}$ will stand for the existence of a constant $C$ such that $g_{1}\leq C g_{2}$.

We mention now some properties on the periodic ivp of (\ref{fnls11}) 
For $s=1$, the main results on well-posedness are derived from the works by Bourgain, \cite{Bourgain1993}, where  local and global well-posedness of the periodic cubic NLS in $L^{2}$ is proved. On the other hand, for $s\in (1/2,1)$, local well-posedness in $H^{\alpha}$, for $\alpha>s_{c}=\frac{1-s}{2}$, is proved in \cite{DemirbasET}, while ill-posedness for $\alpha<s_{c}$ is shown in \cite{ChoHKL2015}. In \cite{ErGT}, smoothing properties of the cubic fNLS are improved, in the sense that the nonlinear part of the solution is smoother than the initial data with a gain of regularity from $\alpha>\frac{3-2s}{4}$ to $\alpha+a$ where $a\leq \min\{2s-1,2\alpha+s-1\}$ and $s\in (1/2,1]$. Finally, local well-posedness of the periodic ivp of (\ref{fnls11}) in $H^{\alpha}, \alpha>0$, for nonlinear terms $f(u)=N(|u|)u$, and two cases of real function $N$ (one of which includes the cubic case considered here) is studied in
\cite{SanchezRR2025} when the initial data satisfies a non-vanishing condition $\inf_{x}|u_{0}(x)|>0$.

\begin{proposition}
\label{propos12a}
The following quantities
\begin{eqnarray}
I_{1}(v,w)&=&\frac{1}{2}\int_{-L}^{L}(v^{2}+w^{2})dx=\frac{1}{2}\int_{-L}^{L}|u|^{2}dx, \label{fnls14a}\\
I_{2}(v,w)&=&\frac{1}{2}\int_{-L}^{L}(vw_{x}-wv_{x})dx=\frac{1}{2}\int_{-L}^{L}{\rm Im}(u\overline{u}_{x})dx, \label{fnls14b}\\
H(v,w)&=&\frac{1}{2}\int_{-L}^{L}\left(\left( (|D|^{s}v)^{2}+(|D|^{s}w)^{2}\right)-V(v,w)\right) d{x}\nonumber\\
&=&\int_{-L}^{L}\left(\frac{1}{2}||D|^{s}u|^{2}-V(u)\right),\label{fnls14c}
\end{eqnarray}
where $u=v+iw$, $|D|^{s}=(-\partial_{xx})^{s/2}$, and $V$ is defined in (\ref{fnls11a}), are preserved by smooth solutions of the periodic ivp of (\ref{fnls11}) on $(-L,L)$. The quantity (\ref{fnls14c}) is the Hamiltonian function of the canonical Hamiltonian structure of (\ref{fnls11}).
\begin{eqnarray}
\frac{d}{dt}\begin{pmatrix}v\\w\end{pmatrix}=J\delta H(v,w),\; J=\begin{pmatrix}0&1\\-1&0\end{pmatrix},\label{fnls15}
\end{eqnarray}
where $\delta H=(\frac{\delta H}{\delta v},\frac{\delta H}{\delta w})^{T}$ denotes the variational (Fr\'echet) derivative of $H$.
\end{proposition}
\begin{proof}
We start checking (\ref{fnls15}), which is not hard to prove by using the formulation (\ref{fnls12}) and Plancherel's identity
\begin{eqnarray*}
\int_{-L}^{L}(|D|^{s}v)(|D|^{s}w)dx&=&\sum_{k=-\infty}^{\infty}\widehat{(|D|^{s}v)}(k)\overline{\widehat{(|D|^{s}w)}(k)}\\
&=&\sum_{k=-\infty}^{\infty}|k|^{2s}\widehat{v}(k)\overline{\widehat{w}(k)}=\int_{-L}^{L}\left((-\partial_{xx})^{s}v\right)wdx.
\end{eqnarray*}
On the other hand, the symplectic matrix $J$ in (\ref{fnls15}) defines the Poisson bracket
\begin{eqnarray*}
\{F,G\}=\int_{-L}^{L}(\delta f)^{T}J(\delta G)dx,
\end{eqnarray*}
for $F=F(v,w), G=G(v,w)$ real-valued functions $F,G:M\rightarrow \mathbb{R}$, where
$$M=\{u=(v,w), v,w\in H^{\infty}(-L,L)\},$$ and depending on $v,w$, and derivatives. Note that
\begin{eqnarray*}
\delta I_{1}(v,w)=\begin{pmatrix}v\\w\end{pmatrix},\;\delta I_{2}(v,w)=\begin{pmatrix}w_{x}\\v_{x}\end{pmatrix}. 
\end{eqnarray*}
Then, using Plancherel's identity and (\ref{fnls11a})
\begin{eqnarray}
\{I_{1},H\}(v,w)&=&\int_{-L}^{L}\left(v(-\partial_{xx})^{s}w-w(-\partial_{xx})^{s}v-v{\rm Im}f(v,w)+w{\rm Re}f(v,w)\right)dx\nonumber\\
&=&\int_{-L}^{L}(wv(v^{2}+w^{2})-vw(v^{2}+w^{2}))dx=0,\label{fnls16}
\end{eqnarray}
where we wrote $f=f(v,w)=(v^{2}+w^{2})\begin{pmatrix}v\\w\end{pmatrix}$. Similarly, integration by parts, Plancherel's identity, property (A0), and periodicity of $v,w$ lead to
\begin{eqnarray}
\{I_{2},H\}(v,w)&=&\int_{-L}^{L}\left(w_{x}(-\partial_{xx})^{s}w+v_{x}(-\partial_{xx})^{s}v-w_{x}w(v^{2}+w^{2})-v_{x}v(v^{2}+w^{2})\right)dx\nonumber\\
&=&-\int_{-L}^{L}(v_{x}v(v^{2}+w^{2})+w_{x}w(v^{2}+w^{2}))dx\nonumber\\
&=&-\frac{1}{2}\int_{-L}^{L}(v_{x}\frac{\partial V}{\partial v}(v,w)+w_{x}\frac{\partial V}{\partial w}(v,w))dx\nonumber\\
&=&-\frac{1}{2}\int_{-L}^{L}\frac{\partial}{\partial x}V(v,w)dx=0.\label{fnls17}
\end{eqnarray}
From (\ref{fnls16}) and (\ref{fnls17}), the quantities (\ref{fnls14a}), (\ref{fnls14b}) have zero Poisson bracket with the Hamiltonian (\ref{fnls14c}); then the proposition follows, cf. \cite{Olver}.
\end{proof}
An additional property of $f$, that will be used in the sequel, is concerned with its regularity. 
\begin{proposition}
\label{propos11}
If $u_{1}, u_{2}\in L^{2}\cap L^{\infty}$, then
\begin{eqnarray}
||f(u_{1})-f(u_{2})||\leq 3(||u_{1}||_{\infty}^{2}+||u_{2}||_{\infty}^{2})||u_{1}-u_{2}||.\label{fnlsL1}
\end{eqnarray}
\end{proposition}
\begin{proof}
We can write
\begin{eqnarray*}
f(u_{1})-f(u_{2})=\int_{0}^{1}f'(\tau u_{1}+(1-\tau)u_{2})(u_{1}-u_{2})d\tau.\label{fnlsL2}
\end{eqnarray*}
Note that each component of $f'(u)$ is bounded by $3|u|^{2}$. Then
\begin{eqnarray}
|f(u_{1})-f(u_{2})|&\leq &\int_{0}^{1}3|\tau u_{1}+(1-\tau)u_{2}|^{2}|u_{1}-u_{2}|d\tau\nonumber\\
&\leq &3(||u_{1}||_{\infty}^{2}+||u_{2}||_{\infty}^{2})|u_{1}-u_{2}|,\label{fnlsL2b}
\end{eqnarray}
from which (\ref{fnlsL1}) follows.
\end{proof}
Using the regularity of $f$ and the fact that $H^{l}, l>1/2$, is an algebra, we have the following extension of Proposition \ref{propos11}, cf. \cite{SanchezRR2025}:
\begin{proposition}
\label{propos12}
Let $l>1/2$, $u_{1}, u_{2}\in L^{2}\cap L^{\infty}$, then there is a constant $C>0$ such that
\begin{eqnarray}
||f(u_{1})-f(u_{2})||_{l}\leq C(||u_{1}||_{l}^{2}+||u_{2}||_{l}^{2})||u_{1}-u_{2}||_{l}.\label{fnlsL1_1}
\end{eqnarray}
\end{proposition}

We now review some literature concerning the numerical approximation of the periodic initial-value problem (\ref{fnls11}), (\ref{fnls13a}) (or (\ref{fnls12}), (\ref{fnls12a})). For the nonfractional case ($s=1$), we refer to \cite{BarlettiBGI2024} for recent techniques about the numerical approximation based on the preservation of the Hamiltonian structure. (See also references therein for a review of previous general approaches, including geometric methods.)
In the fractional case, most of the literature, as in the study of well-posedness, is focused on the nonlinearity $f(u)=|u|^{2\sigma}u, \sigma>0$, with particular emphasis on the cubic case $\sigma=1$. Thus, in \cite{WangH2015}, an energy conservative Crank-Nicolson (CN) finite difference method is introduced and analyzed for the Dirichlet problem with the Riesz space fractional derivative (instead of (\ref{fnls11b})) and $1/2<s<1$. (See also \cite{WangXY2014,WangHZ2016,LiHW2017} for other approaches based on different approximations of the Riesz fractional derivative, linearly implicit conservative finite differences for a coupled fNLS equations, and a Galerkin finite element method for a more general nonlinear fractional Ginzburg-Landau equation, respectively.) For the frational derivative defined in (\ref{fnls11b}) and $f(u)=|u|^{2\sigma}u, \sigma>0$, three Fourier spectral methods are proposed in \cite{DuoZ2016}: split-step (SSFS), Crank-Nicolson (CNFS), and relaxation (ReFS) Fourier spectral schemes. Several conservation properties with respect to the mass (\ref{fnls14a}) and energy (\ref{fnls14c}) are derived. The spectral order of convergence in space and second order of convergence in time are checked computationally. The experiments also show the performance of the methods when simulating different long-term dynamics. On the other hand, the preservation of the symplectic structure and its consequences for long time simulations is in the origin of the methods introduced in \cite{WangH2018} for the cubic case. They are based on a suitable reformulation of the equation, a Fourier pseudospectral approximation in space, and a symplectic (the implicit midpoint rule) and multisymplectic methods as time integrators for the semidiscrete system. SSFS methods are used to investigate the dynamics of plane wave solutions of the cubic fNLS equation in \cite{DuoLZ} (see also references therein), while Fourier spectral methods with time splitting, \cite{BaoM2002,BaoM2003}, or composite Runge-Kutta methods, \cite{Driscoll2002}, are used to study the dynamics of ground state solutions in \cite{KleinSM2014}.
\subsection{Highlights and structure}
The main contributions of the paper are distributed according to the following structure:
\begin{itemize}
\item The semidiscretization of the periodic ivp of (\ref{fnls11}) with a Fourier Galerkin method is analyzed in section \ref{sec2}. The analysis consists of: (i)  the derivation of conserved quantities of the semidiscrete system; (ii) error estimates with respect to the exact solution $u\in H^{\mu}, \mu$ large enough, in the $L^{2}$ and $H^{s}$ norms ($s\in 1/2,1]$) of $O(N^{-\mu})$ and $O(N^{s-\mu})$, respectively (leading to spectral convergence for $u$ smooth) and where $N$  denotes the degree of the trigonometric polynomial approximation; (iii) a result of boundedness for the temporal derivatives of the semidiscrete solution.
\item Section \ref{sec3} is devoted to the full discretization of the semidiscrete system with a family of Runge-Kutta Compositon methods, based on the implicit midpoint rule, and proposed by Yoshida, \cite{Yoshida1990} (see also \cite{HairerLW2004}). Existence and uniqueness of fully discrete solution, as well as some geometric properties, are first discussed. The convergence of the numerical solution is established from the derivation  of error estimates in $L^{2}$ and $H^{s}$ norms ($s\in (1/2,1]$), which depend on the regularity of the exact solution and a priori estimate of the local temporal error. This estimate is proved to hold for the first two methods of the family (the implicit midpoint rule itself and the scheme with three stages) in Appendix \ref{appA}, showing a general procedure for its proof in the case of any method of the family.
\item In section \ref{sec4}, the previous results are illustrated with some numerical experiments for both the non-fractional and fractional case and the method of three stages and order four. After some details of the implementation and taking solitary wave solutions as reference, the performance of the full discretization is numerically checked in accuracy and in some aspects concerning long time simulations, such as the time behaviour of the errors with respect to the invariants of the problem and some parameters of the waves. 
\end{itemize}
This work is motivated by mainly two papers. In the first one, \cite{DD2021}, the same full discretization was analyzed for the numerical approximation of the Korteweg-de Vries (KdV) equation. The approach here reveals some differences in the treatment of the nonlinear term and the conditions for the error estimates and their proofs, and we thought they were worthy of attention. The second paper is \cite{NRAD1}, where a class of solitary wave solutions for the fNLS equation (\ref{fnls11}) with $f(u)=|u|^{2\sigma}u, \sigma>0$, was proved. As a continuation of this work, we are interested in the numerical study of their dynamics, by using the full discretization considered in the present paper. The results obtained here will give us the necessary confidence in the accuracy of the computations made in \cite{NRAD2}.
\section{Spectral semidiscretization in space}
\label{sec2}
\subsection{Presentation and properties of the semidiscretization}
This section is devoted to the analysis of the discretization in space of the periodic ivp
\begin{eqnarray}
&&iu_{t}- (-\partial_{xx})^{s}u+ f(u)=0,\quad {x}\in [-\pi,\pi],\quad 0\leq t\leq T,\nonumber\\
&&u(x,0)=u_{0}(x),\label{fnls2_0}
\end{eqnarray} 
where $T>0$, $0<s\leq 1$, $u_{0}$ is smooth, $2\pi$-periodic, and $f(u)=|u|^{2}u$. Let $N\geq 1$ be an integer and
\begin{eqnarray*}
S_{N}={\rm span}\{e^{ikx}, k\in\mathbb{Z}, -N\leq k\leq N, \, x\in[-\pi,\pi]\},
\end{eqnarray*}
be the finite-dimensional space of trigonometric polynomials of degree at most $N$. The spatial discretization described below will require the orthogonal projection onto $S_{N}$ and some of its properties. For $u\in L^{2}(-\pi,\pi)$, let
\begin{eqnarray*}
\widehat{u}_{k}=\widehat{u}(k)=\frac{1}{2\pi}\int_{-\pi}^{\pi}u(x)e^{-ikx}dx,\; k\in\mathbb{Z},
\end{eqnarray*}
be the $k$th Fourier coefficient of $u$. The $L^{2}$-projection $P_{N}:L^{}(-\pi,\pi)\rightarrow S_{N}$ of $u$ is defined as
\begin{eqnarray}
P_{N}u=\sum_{|k|\leq N}\widehat{u}_{k}e^{ikx}.\label{fnls2_1}
\end{eqnarray}
Some properties of (\ref{fnls2_1}) are listed below, \cite{Mercier,Canuto2006}:
\begin{itemize}
\item[(P1)] $P_{N}$ commutes with $(-\partial_{xx})^{s}, s>0$.
\item[(P2)] If $v\in L^{2}(-\pi,\pi)$ and $\varphi\in S_{N}$ then
\begin{eqnarray*}
(P_{N}v,\varphi)=(v,\varphi).
\end{eqnarray*}
\item[(P3)] For integers $0\leq j\leq \mu$ and $v\in L^{2}(-\pi,\pi)$, the following estimates hold:
\begin{eqnarray}
||v-P_{N}v||_{j}&\lesssim &N^{j-\mu}||v||_{\mu},\; \mu\geq 0,\label{P3}\\
|v-P_{N}v|_{\infty}&\lesssim &N^{1/2-\mu}||v||_{\mu},\; \mu>1/2.\nonumber
\end{eqnarray}
\end{itemize}
We will also make use of some inverse inequalities on $S_{N}$, \cite{Mercier,Canuto2006}: For integers $0\leq j\leq \mu$ and $\varphi\in S_{N}$
\begin{eqnarray}
||\varphi||_{\mu}\lesssim N^{\mu-j}||\varphi||_{j},\; |\varphi|_{\infty}\lesssim N^{1/2}||\varphi||.\label{Inverse}
\end{eqnarray}

Note also that if $v\in S_{N}$ and $0<s\leq 1$, then $(-\partial_{xx})^{s}v\in S_{N}$. In addition, if
\begin{eqnarray*}
v(x)=\sum_{|k|\leq N}\widehat{v}_{k}e^{ikx},
\end{eqnarray*}
then
\begin{eqnarray*}
(-\partial_{xx})^{s}v(x)=\sum_{|k|\leq N}\widehat{w}_{k}e^{ikx},\; \widehat{w}_{k}=|k|^{2s}\widehat{v}_{k},\; -N\leq k\leq N.
\end{eqnarray*}

The semidiscrete Fourier-Galerkin approximation to the solution of (\ref{fnls2_0}) is defined as a complex-valued map $u^{N}:[0,T]\rightarrow S_{N}$ such that, for all $\varphi\in S_{N}$
\begin{eqnarray}
&&(iu_{t}^{N}-(-\partial_{xx})^{s}u^{N}+f(u^{N}),\varphi)=0,\; 0\leq t\leq T,\label{fnls2_2}\\
&&u^{N}(x,0)=P_{N}u_{0}(x),\;x\in [-\pi,\pi].\nonumber
\end{eqnarray}
Note that, if $F:S_{N}\rightarrow S_{N}$ is defined as
\begin{eqnarray*}
(F(u),\varphi)=(-i(-\partial_{xx})^{s}u+if(u),\varphi),\; \varphi\in S_{N},
\end{eqnarray*}
or, equivalently
\begin{eqnarray}
F(u)=-i(-\partial_{xx})^{s}u+iP_{N}f(u),\label{fnls2_4*}
\end{eqnarray}
then, (\ref{fnls2_2}) can be written as
\begin{eqnarray}
&&\frac{d}{dt}u^{N}=F(u^{N}),\;0\leq t\leq T,\label{fnls2_3}\\
&&u^{N}(0)=P_{N}u_{0}.\nonumber
\end{eqnarray}
Note that
\begin{eqnarray*}
(F(u),u)=-i\int_{-\pi}^{\pi}||D|^{s}u|^{2}dx+i\int_{-\pi}^{\pi}|u|^{4}dx,
\end{eqnarray*}
so 
\begin{eqnarray}
{\rm Re}(F(u),u)=0.\label{fnls2_4**}
\end{eqnarray}
Alternatively, (\ref{fnls2_2}) can be introduced as a real system for real-valued functions $v^{N},w^{N}:[0,T]\rightarrow S_{N}$ such that, for all $\varphi,\psi\in S_{N}$
\begin{eqnarray*}
&&(v_{t}^{N}-(-\partial_{xx})^{s}w^{N}+{\rm Im}f(v^{N},w^{N}),\varphi)=0,\nonumber\\
&&(-w_{t}^{N}-(-\partial_{xx})^{s}v^{N}+{\rm Re}f(v^{N},w^{N}),\psi)=0,\; 0\leq t\leq T,\label{fnls2_4}\\
&&v^{N}(x,0)=P_{N}v_{0}(x),\; w^{N}(x,0)=P_{N}w_{0}(x)\;x\in [-\pi,\pi],\nonumber
\end{eqnarray*}
where $u_{0}=v_{0}+iw_{0}$. The real version of (\ref{fnls2_3}) will therefore be
\begin{eqnarray}
\frac{d}{dt}\begin{pmatrix}v^{N}\\w^{N}\end{pmatrix}=\begin{pmatrix}F_{1}(v^{N},w^{N})\\F_{2}(v^{N},w^{N})\end{pmatrix},\label{fnls2_5}
\end{eqnarray}
where
\begin{eqnarray*}
F_{1}(v,w)&=&(-\partial_{xx})^{s}w-{\rm Im}f(v,w),\\
F_{2}(v,w)&=&-(-\partial_{xx})^{s}v+{\rm Re}f(v,w).
\end{eqnarray*}
The semidicrete problem (\ref{fnls2_3}) has the following conserved quantities, \cite{QuS2019}.
\begin{proposition}
\label{propos21}
While the solution $u^{N}$ of (\ref{fnls2_2}) exists, it holds that
\begin{eqnarray*}
\frac{d}{dt}I_{1}(u^{N})=\frac{d}{dt}I_{2}(u^{N})=\frac{d}{dt}H(u^{N})=0.
\end{eqnarray*}
\end{proposition}
\begin{proof}
Taking $\varphi=u^{N}$ in (\ref{fnls2_2}) leads to
\begin{eqnarray}
i(u_{t}^{N},u^{N})-((-\partial_{xx})^{s}u^{N},u^{N})=(f(u^{N}),u^{N}).\label{fnls2_6}
\end{eqnarray}
Note that, from Parseval's identity, 
\begin{eqnarray*}
((-\partial_{xx})^{s}u^{N},u^{N})=\sum_{k=-\infty}^{\infty}|k|^{2s}|\widehat{u}^{N}(k)|^{2},
\end{eqnarray*}
and
\begin{eqnarray*}
(f(u^{N}),u^{N})=\int_{-\pi}^{\pi}|u^{N}|^{4}dx.
\end{eqnarray*}
Therefore, taking the imaginary part in (\ref{fnls2_6}), we have ${\rm Im}(i(u_{t}^{N},u^{N}))=0$. Observe that
\begin{eqnarray*}
{\rm Im}(i(u_{t}^{N},u^{N}))={\rm Re}(u_{t}^{N},u^{N})=\frac{1}{2}\frac{d}{dt}||u^{N}||^{2},
\end{eqnarray*}
which implies 
\begin{eqnarray*}
\frac{d}{dt}I_{1}(u^{N})=0.
\end{eqnarray*}
On the other hand, taking $\varphi=u_{x}^{N}$ in (\ref{fnls2_2}) leads to
\begin{eqnarray}
i(u_{t}^{N},u_{x}^{N})-((-\partial_{xx})^{s}u^{N},u_{x}^{N})=(f(u^{N}),u_{x}^{N}).\label{fnls2_7}
\end{eqnarray}
Note that, from Parseval's identity, 
\begin{eqnarray*}
((-\partial_{xx})^{s}u^{N},u_{x}^{N})=-ik\sum_{k=-\infty}^{\infty}|k|^{2s}|\widehat{u}^{N}(k)|^{2},
\end{eqnarray*}
and, from (\ref{fnls11a}) and periodicity of $u^{N}$
\begin{eqnarray*}
(f(u^{N}),u_{x}^{N})&=&I_{1}+iI_{2},\\
I_{1}&=&\int_{-\pi}^{\pi}\left(v_{x}^{N}{\rm Re}f(v^{N},w^{N})+w_{x}^{N}{\rm Im}f(v^{N},w^{N})\right)dx\\
&=&\frac{1}{2}\int_{-\pi}^{\pi}\frac{d}{dx}V(v^{N},w^{N})dx=0.
\end{eqnarray*}
Therefore, taking the real part in (\ref{fnls2_7}), we have ${\rm Re}(i(u_{t}^{N},u_{x}^{N}))=0$. Observe that
\begin{eqnarray*}
{\rm Re}(i(u_{t}^{N},u_{x}^{N}))&=&-{\rm Im}(u_{t}^{N},u_{x}^{N})=\frac{1}{2i}\int_{-\pi}^{\pi}(u_{t}^{N}\overline{u_{x}^{N}}-\overline{u_{t}^{N}}u_{x}^{N})dx\\
&=&\frac{1}{2i}\int_{-\pi}^{\pi}(u_{t}^{N}\overline{u_{x}^{N}}+\overline{u_{xt}^{N}}u^{N})dx\\
&=&\frac{1}{2i}\frac{d}{dt}\int_{-\pi}^{\pi}u^{N}\overline{u_{x}^{N}}dx=\frac{1}{2i}\frac{d}{dt}(u^{N},u_{x}^{N}),
\end{eqnarray*}
which implies 
\begin{eqnarray*}
\frac{d}{dt}I_{2}(u^{N})=0.
\end{eqnarray*}
If, finally, we take $\varphi=u_{t}^{N}$ in (\ref{fnls2_2}) then we have
\begin{eqnarray}
i(u_{t}^{N},u_{t}^{N})-((-\partial_{xx})^{s}u^{N},u_{t}^{N})=(f(u^{N}),u_{t}^{N}).\label{fnls2_8}
\end{eqnarray}
Note that
\begin{eqnarray*}
\frac{d}{dt}(|D|^{s}u^{N},|D|^{s}u^{N})=2{\rm Re}\int_{-\pi}^{\pi}|D|^{s}u^{N}\overline{|D|^{s}u_{t}^{N}}dx=2{\rm Re}((-\partial_{xx})^{s}u^{N},u_{t}^{N}),\label{fnls2_9}
\end{eqnarray*}
and, from (\ref{fnls11a}) 
\begin{eqnarray}
(f(u^{N}),u_{t}^{N})&=&I_{1}+iI_{2},\\
I_{1}&=&\int_{-\pi}^{\pi}\left(v_{t}^{N}{\rm Re}f(v^{N},w^{N})+w_{t}^{N}{\rm Im}f(v^{N},w^{N})\right)dx\nonumber\\
&=&\frac{1}{2}\int_{-\pi}^{\pi}\frac{d}{dt}V(v^{N},w^{N})dx.\label{fnls2_10}
\end{eqnarray}
Then, from (\ref{fnls2_9}), (\ref{fnls2_10}), and taking real part in (\ref{fnls2_8}), it holds that
\begin{eqnarray*}
\frac{d}{dt}H(u^{N})=0.
\end{eqnarray*}
\end{proof}
A key property for the next sections is given by the following result.
\begin{lemma}
\label{lemmaL5}
Let $s>1/2$,
$$C_{\infty}=\left(\sum_{k=-\infty}^{\infty}\frac{1}{1+|k|^{2s}}\right)^{1/2},$$ and let $u_{0}$ be such that the solution $u^{N}$ of  (\ref{fnls2_3}) with $I_{1}=I_{1}(u^{N}(0))$ and $E=E(u^{N}(0))$, satisfies
\begin{eqnarray}
C_{*}:=1-{C_{\infty}^{2}}I_{1}\geq 0,\; \frac{I_{1}}{2}+E\geq 0.\label{fnlsL4}
\end{eqnarray}
Then, while $u^{N}$ exists, it holds that
\begin{eqnarray}
||u^{N}(t)||_{s}\leq C_{S}:=\left(\frac{1}{C_{*}}(2E+I_{1})\right)^{1/2}.\label{fnlsL5}
\end{eqnarray}
\end{lemma}
\begin{proof}
Note first that 
\begin{eqnarray}
|u^{N}|_{\infty}&\leq&\sum_{|k|\leq N}|\widehat{u^{N}}(k)|=\sum_{|k|\leq N}(1+|k|^{s})\frac{|\widehat{u^{N}}(k)|}{(1+|k|^{s})}\nonumber\\
&\leq& \left(\sum_{|k|\leq N}\frac{1}{(1+|k|^{s})^{2}}\right)^{1/2}||u^{N}||_{s}\leq C_{\infty}||u^{N}||_{s}.\label{fnlsL6}
\end{eqnarray}
On the other hand, using Proposition \ref{propos21}, we write
\begin{eqnarray}
\int_{-\pi}^{\pi}\frac{1}{2}||D|^{s}u^{N}|^{2}dx=E(u^{N})+\int_{-\pi}^{\pi}V(u^{N})dx,\label{fnlsL7}
\end{eqnarray}
where, from (\ref{fnls11a}) and (\ref{fnlsL6})
\begin{eqnarray}
\int_{-\pi}^{\pi}V(u^{N})dx\leq \frac{|u^{N}|_{\infty}^{2}}{2}I_{1}\leq \left(\frac{C_{\infty}||u^{N}||_{s}}{\sqrt{2}}\right)^{2}I_{1}.\label{fnlsL8}
\end{eqnarray}
Applying (\ref{fnlsL8}) to (\ref{fnlsL7}), we have
\begin{eqnarray*}
\frac{1}{2}||u^{N}||_{s}^{2}\leq \left(\frac{C_{\infty}||u^{N}||_{s}}{\sqrt{2}}\right)^{2}I_{1}+E+\frac{1}{2}I_{1},
\end{eqnarray*}
which, from (\ref{fnlsL4}), implies (\ref{fnlsL5}).
\end{proof}
\subsection{Convergence of the semidiscretization}
\label{sec22}
\begin{theorem}
\label{th22}
Let $u^{N}$ be the solution of (\ref{fnls2_2}) satisfying the hypotheses of Lemma \ref{lemmaL5}. Assume that
 the solution $u$ of (\ref{fnls2_0}) is in $H^{\mu}, \mu> s$, for $t\in [0,T]$. Then
\begin{eqnarray}
\max_{0 \leq t\leq T}||u^{N}-u||_{s}\lesssim N^{s-\mu}.\label{fnls2_10b}
\end{eqnarray}
\end{theorem}
\begin{proof}
We write $u^{N}=u+\theta+\rho$, with
$$\theta=u^{N}-P_{N}u\in S_{N},\; \rho=P_{N}u-u.$$ Then, for $\varphi\in S_{N}$, it holds that
\begin{eqnarray}
(i\theta_{t},\varphi)-((-\partial_{xx})^{s}\theta,\varphi)&=&(iu^{N}_{t}-(-\partial_{xx})^{s}u^{N},\varphi)-(iu_{t}-(-\partial_{xx})^{s}u,\varphi)\nonumber\\
&=&(f(u)-f(u^{N}),\varphi).\label{fnls2_11}
\end{eqnarray}
We consider $\varphi=\theta+(-\partial_{xx})^{s}\theta$ in (\ref{fnls2_11}) and take imnaginary parts. Using that
\begin{eqnarray*}
{\rm Im}((-\partial_{xx})^{s}\theta,(-\partial_{xx})^{s}\theta)=0,
\end{eqnarray*}
Lemma \ref{lemmaL5} and (\ref{fnlsL1_1}) in Proposition \ref{propos12}, we have
\begin{eqnarray*}
\frac{1}{2}\frac{d}{dt}||\theta||_{s}^{2}&=&{\rm Im}(f(u)-f(u^{N}),\theta+(-\partial_{xx})^{s}\theta)\\
&=&{\rm Im}(f(u)-f(u^{N}),\theta)+{\rm Im}(|D|^{s}(f(u)-f(u^{N})),|D|^{s}\theta)\\
&\leq &C||u-u^{N}||_{s}||\theta||_{s},
\end{eqnarray*}
from some constant $C$ independent of $N$. From (\ref{P3}) there holds
\begin{eqnarray*}
\frac{1}{2}\frac{d}{dt}||\theta||_{s}^{2}&\leq&C||\theta+\rho||_{s}||\theta||_{s}\leq C(||\theta||_{s}^{2}+||\rho||_{s}||\theta||_{s})\\
&\leq & C\left(||\theta||_{s}^{2}+N^{s-\mu}||u||_{\mu}||\theta||_{s}\right)\\
&\leq & C(||\theta||_{s}^{2}+N^{2(s-\mu)}),
\end{eqnarray*} 
for some constant $C$ and $0\leq t\leq T$. Therefore, using Gronwall's lemma and that $\theta(0)=0$, (\ref{fnls2_10b}) follows.
\end{proof}
\begin{theorem}
\label{th22b}
Let $u^{N}$ be the solution of (\ref{fnls2_2}) satisfying the hypotheses of Lemma \ref{lemmaL5}. Assume that
 the solution $u$ of (\ref{fnls2_0}) is in $H^{\mu}, \mu> 0$, for $t\in [0,T]$. 
Then
\begin{eqnarray}
\max_{0 \leq t\leq T}||u^{N}-u||\lesssim N^{-\mu},\label{fnls2_10bb}
\end{eqnarray}
\end{theorem}
\begin{proof}
With the same arguments as in Theorem \ref{th22}, but taking $\varphi=\theta$ in (\ref{fnls2_11}), from (\ref{fnlsL2b}) in Proposition \ref{propos11} and Lemma \ref{lemmaL5} we have
\begin{eqnarray*}
\frac{1}{2}\frac{d}{dt}||\theta||^{2}&=&{\rm Im}(f(u)-f(u^{N}),\theta)\\
&\leq &C||u-u^{N}||||\theta||,
\end{eqnarray*}
from some constant $C$ independent of $N$. Then, (\ref{P3}) implies the $L^{2}$ estimate (\ref{fnls2_10bb}).
\end{proof}
As a consequence, we have the following result of boundedness for the derivatives of $u_{t}^{N}$, that will be used in section \ref{sec3}.
\begin{proposition}
\label{propos22}
Let su assume the hypotheses of Theorem \ref{th22}. Given nonnegative integers $j$ and $l$, and provided $\mu\geq \max\{2, 2sj+l\}$, there is a constant $C$, independent of $N$, such that
\begin{eqnarray}
\max_{0\leq t\leq T}||\partial_{t}^{j}u^{N}||_{l}\leq C.\label{fnls2_13}
\end{eqnarray}
\end{proposition}

\begin{proof}
Provided $\mu\geq \max\{2,l\}$, using (\ref{P3}), (\ref{Inverse}), (\ref{fnls2_10bb}), we have
\begin{eqnarray*}
||u^{N}||_{l}&\leq &||u-P_{N}u||_{l}+||P_{N}u-u^{N}||_{l}+||u||_{l}\\
&\lesssim&N^{l-\mu}||u||_{\mu}+N^{l}\left(||u-P_{N}u||+||u-u^{N}||\right)+||u||_{l}\\
&\lesssim&N^{l-\mu}||u||_{\mu}+N^{l-\mu}||u||_{\mu}+||u||_{l},
\end{eqnarray*}
which proves (\ref{fnls2_13}) when $j=0$. 
Note now that, from (\ref{fnls2_3})
\begin{eqnarray}
||\partial_{t}u^{N}||_{l}&\leq&||u^{N}||_{l+2s}+||P_{N}f(u^{N})||_{l}\nonumber\\
&\leq&||u^{N}||_{l+2s}+||f(u^{N})||_{l}.\label{fnls2_14}
\end{eqnarray}
Since $f(0)=0$, using Proposition \ref{propos12}, (\ref{fnls2_13}) when $j=0$, and (\ref{fnls2_14}), we have (\ref{fnls2_13}) when $j=1$ and $\mu\geq \max\{2, 2s+l\}$.
%
%
The argument above can be used recursively, by differentiating (\ref{fnls2_3}) $j-1$ times with respect to $t$ in order to bound the time derivatives of $f(u^{N})$ in a similar way.
\end{proof}

\section{Full discretization}
\label{sec3}
\subsection{Fully discrete scheme. First properties}
\label{sec31}
In this section we will analyze the fully discrete scheme based on approximating the ivp of the semidiscrete system (\ref{fnls2_3}) with the family of singly diagonally implicit RK methods
 \begin{equation}\label{44}
\begin{tabular}{c | c}
& $a_{ij}$\\ \hline
 & $b_{i}$
\end{tabular}=
\begin{tabular}{c | ccccc}
& $b_{1}/2$ & & &&\\
&  $b_{1}$ & $b_{2}/2$ &&&\\
&$b_{1}$&$b_{2}$&$\ddots$&&\\
&$\vdots$&$\vdots$ &&$\ddots$&\\
&$b_{1}$&$b_{2}$&$\cdots$&$\cdots$&$b_{q}/2$\\ \hline
 & $b_{1}$&$b_{2}$&$\cdots$&$\cdots$&$b_{q}$
\end{tabular},
\end{equation}
for nonzero real numbers $b_{j}, j=1,\ldots,q$, $q$ of the form $q=3^{p-1}, p\geq 1$, proposed by Yoshida, \cite{Yoshida1990}. The scheme (\ref{44}) can be formulated as a $q$-stage RK-composition method based on the implicit midpoint rule (IMR), which is the first element of the family (with $q=1$), cf. e.~g. \cite{HairerLW2004}. Symplecticity and symmetry properties of (\ref{44}) are obtained from the general construction in \cite{Yoshida1990}, applied to the IMR. Note also that some of the $b_{j}$ may be negative, which prevents (\ref{44}) to be A-stable, although it is absolutely stable in a strip of finite width in ${\rm Re}z\leq 0$. A summary of properties of (\ref{44}), collected from several references therein, can be checked in \cite{DD2021}. In particular, the $q$-stage method has order of accuracy $2p$. 

For $1\leq j\leq q$, let $\tau_{n}^{j}=t_{n}+k\sum_{m=1}^{j}b_{m}$. Note that, since some of the $b_{j}$ may be negative and some $\tau_{n}^{j}$ may overcome $t_{n+1}$, then, applying the reversibility for $t<0$ iof the fNLS equation, we may have to extend well-posedness of (\ref{fnls2_0}) and the definition of the semidiscrete solution (\ref{fnls2_2}) for intervals of the form $[-k,T+k]$, where (\ref{fnls2_10b}) and (\ref{fnls2_13}) hold, cf. \cite{DD2021} and Appendix \ref{appA}.

For the semidiscrete system (\ref{fnls2_3}), we will use the formulation of the resulting fully discrete scheme as a RK-Composition method with IMR, namely
\begin{eqnarray}
Y^{n,0}&=&U^{n},\nonumber\\
Y^{n,j}&=&Y^{n,j-1} +{kb_{j}}{F}\left(\frac{Y^{n,j}+Y^{n,j-1}}{2}\right),\; 1\leq j\leq q,\nonumber\\
U^{n+1}&=&Y^{n,q},\label{fnls3_1}
\end{eqnarray} 
and
for given $U^{n}\in S_{N}, 0\leq n\leq M-1$, approximation to $u^{N}(t_{n}), t_{n}=nk$, with $k>0$, $M$ an integer such that $T=Mk$,
$Y^{n,j}\in S_{N}, 1\leq j\leq q$, and $U^{0}=P_{N}u_{0}$. 

The following result is a consequence of considering the IMR as base scheme for the composition and its geometric properties.
\begin{lemma}
\label{lemma31} Provided $U^{n}, Y^{n,j}\in S_{N}, 1\leq j\leq q$ exist for $1\leq n\leq M$, then it holds that
\begin{eqnarray}
||U^{n}||&=&||U^{0}||,\; 0\leq n\leq M,\label{fnls3_2a}\\
{\rm Im}(U^{n},U_{x}^{n})&=&{\rm Im}(U^{0},U_{x}^{0}),\; 0\leq n\leq M.\label{fnls3_2b}
\end{eqnarray}
\end{lemma}
\begin{proof}
From (\ref{fnls3_1})
\begin{eqnarray}
(Y^{n,j}-Y^{n,j-1},Y^{n,j}+Y^{n,j-1})=2kb_{j}\left(F\left(\frac{Y^{n,j}+Y^{n,j-1}}{2}\right),\frac{Y^{n,j}+Y^{n,j-1}}{2}\right).\label{fnls3_3}
\end{eqnarray}
Note now that
\begin{eqnarray*}
{\rm Re}(Y^{n,j}-Y^{n,j-1},Y^{n,j}+Y^{n,j-1})=||Y^{n,j}||^{2}-||Y^{n,j-1}||^{2},
\end{eqnarray*}
and, from (\ref{fnls2_4**})
\begin{eqnarray*}
{\rm Re}\left(F\left(\frac{Y^{n,j}+Y^{n,j-1}}{2}\right),\frac{Y^{n,j}+Y^{n,j-1}}{2}\right)=0.
\end{eqnarray*}
Therefore, taking real part on both sides of (\ref{fnls3_3}) leads to $||Y^{n,j}||=||Y^{n,j-1}||,1\leq j\leq q$. Since $U^{n+1}=Y^{n,q}$ and $Y^{n,0}=U^{n}$, then (\ref{fnls3_2a}) follows. On the other hand, again from (\ref{fnls3_1})
\begin{eqnarray}
(Y^{n,j}-Y^{n,j-1},Y_{x}^{n,j}+Y_{x}^{n,j-1})=2kb_{j}\left(F\left(\frac{Y^{n,j}+Y^{n,j-1}}{2}\right),\frac{Y_{x}^{n,j}+Y_{x}^{n,j-1}}{2}\right).\label{fnls3_4}
\end{eqnarray}
Using (\ref{fnls2_4*}), note that
\begin{eqnarray*}
\left((-\partial_{xx})^{s}\left(\frac{Y^{n,j}+Y^{n,j-1}}{2}\right),\frac{Y_{x}^{n,j}+Y_{x}^{n,j-1}}{2}\right)\in i\mathbb{R};
\end{eqnarray*}
from (\ref{fnls11a}) and the periodicity of $Y^{n,j/2}:=\frac{Y^{n,j}+Y^{n,j-1}}{2}$, it holds that
\begin{eqnarray*}
\left(f\left(\frac{Y^{n,j}+Y^{n,j-1}}{2}\right),\frac{Y_{x}^{n,j}+Y_{x}^{n,j-1}}{2}\right)&=&\int_{-\pi}^{\pi}\left({\rm Re}f(Y^{n,j/2}){\rm Re}(Y_{x}^{n,j/2})\right.\\
&&\left.+{\rm Im}f(Y^{n,j/2}){\rm Im}(Y_{x}^{n,j/2})\right)dx\\
&=&\frac{1}{2}\int_{-\pi}^{\pi}\frac{d}{dx}V(Y^{n,j/2})dx=0;
\end{eqnarray*}
and
\begin{eqnarray*}
(Y^{n,j}-Y^{n,j-1},Y_{x}^{n,j}+Y_{x}^{n,j-1})&=&(Y^{n,j},Y_{x}^{n,j})-(Y^{n,j-1},Y_{x}^{n,j-1})\\
&&+2{\rm Re}(Y^{n,j},Y_{x}^{n,j-1}).
\end{eqnarray*}
Therefore, taking imaginary parts in (\ref{fnls3_4}), it holds that
\begin{eqnarray*}
{\rm Im}(Y^{n,j},Y_{x}^{n,j})={\rm Im}(Y^{n,j-1},Y_{x}^{n,j-1}),\; 1\leq j\leq q,
\end{eqnarray*}
which implies (\ref{fnls3_2b}).
\end{proof}
The existence of fully discrete solutions is analyzed in the following result.
\begin{proposition}
\label{propos32}
Given $U^{n}\in S_{N}$, the system (\ref{fnls3_1}) admits a solution $Y^{n,j}\in S_{N},1\leq j\leq q$.
\end{proposition}
\begin{proof}
The proof is similar to that in \cite{DD2021} for the KdV equation, and it is based on the application of Brouwer's fixed-point theorem. For this purpose, we consider the alternative version  of (\ref{fnls3_1}) as a real system from the semidiscrete equations given by (\ref{fnls2_5}) and with inner product
\begin{eqnarray*}
\langle\begin{pmatrix}V_{1}\\W_{1}\end{pmatrix},\begin{pmatrix}V_{2}\\W_{2}\end{pmatrix}\rangle={\rm Re}(U_{1},U_{2}),\; U_{j}=V_{j}+iW_{j}, j=1,2.
\end{eqnarray*}
We define $\widetilde{G}:S_{N}\times S_{N}\rightarrow S_{N}\times S_{N}$ as
\begin{eqnarray*}
\widetilde{G}\begin{pmatrix}V\\W\end{pmatrix}=\begin{pmatrix}V\\W\end{pmatrix}-\begin{pmatrix}V^{n}\\W^{n}\end{pmatrix}-\frac{kb_{1}}{2}\begin{pmatrix}F_{1}(V,W)\\F_{2}(V,W)\end{pmatrix},
\end{eqnarray*}
where $U^{n}=V^{n}+iW^{n}$. From (\ref{fnls2_4**}), it holds that
\begin{eqnarray*}
\langle \widetilde{G}\begin{pmatrix}V\\W\end{pmatrix},\begin{pmatrix}V\\W\end{pmatrix}\rangle&=&\left|\left|\begin{pmatrix}V\\W\end{pmatrix}\right|\right|^{2}-\langle\begin{pmatrix}V^{n}\\W^{n}\end{pmatrix},\begin{pmatrix}V\\W\end{pmatrix}\rangle\\
&\geq&\left|\left|\begin{pmatrix}V\\W\end{pmatrix}\right|\right|\left(\left|\left|\begin{pmatrix}V\\W\end{pmatrix}\right|\right|-\left|\left|\begin{pmatrix}V^{n}\\W^{n}\end{pmatrix}\right|\right|\right).
\end{eqnarray*}
So if $$\left|\left|\begin{pmatrix}V\\W\end{pmatrix}\right|\right|=\left|\left|\begin{pmatrix}V^{n}\\W^{n}\end{pmatrix}\right|\right|,$$ then
$$\langle \widetilde{G}\begin{pmatrix}V\\W\end{pmatrix},\begin{pmatrix}V\\W\end{pmatrix}\rangle\geq 0.$$ Since $\widetilde{G}$ is continuous in $S_{N}\times S_{N}$, then Brower's fixed-point theorem applies, \cite{BonaDKM1995}, and then there exists $Z\in S_{N}\times S_{N}$ with $||Z||=||U^{n}||$ and such that $\widetilde{G}(Z)=0$, that is $Z-U^{N}=\frac{kb_{1}}{2}F(Z)$. This implies the existence of $Y^{n,1}$. The proof of existence of $Y^{n,i}, 2\leq i\leq s$ satisfying (\ref{fnls3_1}) follows recursively with an analogous argument.
\end{proof}
The following estimate will be used to establish uniqueness of solution of (\ref{fnls3_1}) and  in the proof of convergence in section \ref{sec32}.

\begin{lemma}
\label{lemma32}
Let $R=||U^{0}||$. If
\begin{eqnarray}
3R^{2}kN\max_{1\leq j\leq q}|b_{j}|<1,\label{fnlsX2}
\end{eqnarray}
then the solutions $Y_{n,j}, 0\leq j\leq q$, of (\ref{fnls3_1}) are unique. Furthermore, under the conditions of Lemma \ref{lemmaL5} and for $k$ small enough, if $U^{0}=u^{N}(0)$, then there is a constant $C_{U}$, independent of $N$ such that
\begin{eqnarray}
\max_{0\leq n\leq M}||U^{n}||_{s}\leq C_{U}.\label{fnlsX2b}
\end{eqnarray}
\end{lemma}
\begin{proof}
Let $Z_{j}\in S_{N}, j=1,2,$ be two solutions of the first equation of the first equation of (\ref{fnls3_1}). Then, from Lemma \ref{lemma31},  $||Z_{1}||=||Z_{2}||=||U^{n}||=R$, and
\begin{eqnarray}
Z_{1}-Z_{2}=kb_{1}\left(F\left(\frac{Z_{1}+U^{n}}{2}\right)-F\left(\frac{Z_{2}+U^{n}}{2}\right)\right),\label{fnls3_5a}
\end{eqnarray}
where $F$ is given by (\ref{fnls2_4*}). Taking the real part of the inner product of (\ref{fnls3_5a}) with $Z_{1}-Z_{2}$ and after some computations,  we have
\begin{eqnarray*}
||Z_{1}-Z_{2}||^{2}&=&kb_{1}\left({\rm Re}\left(\frac{-i}{2}(-\partial_{xx})^{s}(Z_{1}-Z_{2}),Z_{1}-Z_{2}\right)\right.\\
&&\left.+{\rm Re}\left(\left(i f\left(\frac{Z_{1}+U^{n}}{2}\right)-i f\left(\frac{Z_{2}+U^{n}}{2}\right)\right),Z_{1}-Z_{2}\right)\right).
\end{eqnarray*}
Using (\ref{fnlsL2b}) and (\ref{Inverse}),  we have
\begin{eqnarray*}
||Z_{1}-Z_{2}||\leq \frac{3}{4}kN|b_{1}|\left(||Z_{1}||^{2}+||Z_{2}||^{2}+2||U^{n}||^{2}\right)||Z_{1}-Z_{2}||.
\end{eqnarray*}
Then, from (\ref{fnls3_2a}), condition (\ref{fnlsX2}) implies uniqueness of $Y^{n,1}$. The argument can be used recursively to prove uniqueness for the rest of the $Y^{n,j}, 2\leq j\leq q$.

We now prove the second part of the lemma. Note first that the operator
$$\mathcal{A}_{1}=(1-ik\frac{b_{1}}{2}(-\partial_{xx})^{s})^{-1}$$ maps $H^{s}$ into $H^{3s}$ and its norm is less than or equals one. We consider the mapping $G_{0}:S_{N}\rightarrow S_{N}$ given by
\begin{eqnarray*}
G_{0}(Z)=\mathcal{A}_{1}\left(U^{0}+ik\frac{b_{1}}{2}P_{N}f(Z)\right).
\end{eqnarray*}
Let $R_{0}:=C_{S}$ be given in (\ref{fnlsL5}) and $M_{0}>R_{0}$. Due to Proposition \ref{propos12} and since the norm of $\mathcal{A}_{1}$ is less than or equals one, there is a constant $C=C(M_{0})$ such that if $Z_{j}\in S_{N}, ||Z_{j}||_{s}\leq M_{0}, j=1,2$, then
\begin{eqnarray*}
||G_{0}(Z_{1})-G_{0}(Z_{2})||_{s}\leq \frac{k}{2}C||Z_{1}-Z_{2}||_{s},
\end{eqnarray*}
and, consequently, for $k$ sufficiently small, $G_{0}$ is contractive. Since $f(0)=0$ and from the hypothesis and the choice of $M_{0}$, we have
\begin{eqnarray*}
||G_{0}(0)||_{s}\leq ||U^{0}||_{s}<M_{0},
\end{eqnarray*}
and then there is a unique $Z_{0}\in S_{N}$ with $||Z_{0}||\leq M_{0}$ and such that $G_{0}(Z_{0})=Z_{0}$. If we define $Y^{0,1}=2Z_{0}-U^{0}$, then there is a constant $C_{1}$, independent of $N$, such that $||Y^{0,1}||_{s}\leq R_{1}$. By considering
\begin{eqnarray*}
&&\mathcal{A}_{2}=(1-ik\frac{b_{2}}{2}(-\partial_{xx})^{s})^{-1},\\
&&G_{1}:S_{N}\rightarrow S_{N},\; G_{1}(Z)=\mathcal{A}_{2}\left(Y^{0,1}+ik\frac{b_{2}}{2}P_{N}f(Z)\right),
\end{eqnarray*}
and $M_{1}>R_{1}$, the same reasoning proves the existence of $Y^{0,2}\in S_{N}$ such that $||Y^{0,2}||_{s}\leq R_{2}$, for some constant $R_{2}$ independent of $N$. Then from the operators
\begin{eqnarray*}
&&\mathcal{A}_{j}=(1-ik\frac{b_{j}}{2}(-\partial_{xx})^{s})^{-1},\\
&&G_{j-1}:S_{N}\rightarrow S_{N},\; G_{j-1}(Z)=\mathcal{A}_{j}\left(Y^{0,1}+ik\frac{b_{2}}{2}P_{N}f(Z)\right),\; 3\leq j\leq q,
\end{eqnarray*}
we can use the same arguments for all the internal stages and obtain the existence of a constant $C_{1}$, independent of $N$, such that $||U^{1}||_{s}=||Y^{0,q}||_{s}\leq C_{1}$, completing the first step of the recurrence process and showing the way to prove the general step from $U^{n}$ to $U^{n+1}$ and (\ref{fnlsX2b}).
\end{proof}
\subsection{Convergence of the fully discrete scheme}
\label{sec32}
In this section, the question of uniqueness of solution of (\ref{fnls3_1}) as well as a $L^{2}$ and $H^{s}, s\in (1/2,1]$, error estimates with respect to the solution $u$ of (\ref{fnls2_0}) will be analyzed. Let $u^{N}$ be the solution of the semidiscrete system (\ref{fnls2_3}). Let $Z^{n}=u^{N}(t_{n}), 0\leq n\leq M$, and for $0\leq j\leq q, 0\leq n\leq M-1$, let $Z^{n,j}\in S_{N}$ be defined as
\begin{eqnarray}
Z^{n,0}&=&Z^{n},\nonumber\\
Z^{n,j}&=&Z^{n,j-1} +{kb_{j}}{F}\left(\frac{Z_{n}^{j}+Z^{n,j-1}}{2}\right),\; 1\leq j\leq q,\label{fnls3_6}
\end{eqnarray} 
where $F$ is given by (\ref{fnls2_4*}). The local temporal error of (\ref{fnls3_1}) at $t_{n}$, $\Theta^{n}\in S_{N}$, is defined as
\begin{eqnarray}
\Theta^{n}=Z^{n+1}-Z^{n,s}=u^{N}(t_{n+1})-Z^{n,s}.\label{fnls3_6a}
\end{eqnarray}
Due to Proposition \ref{propos11}, consistency of the method will require $L^{\infty}$ bounds of the $Z_{n}^{j}$. This is analyzed in the following result.
\begin{lemma}
\label{lemma33a}
Let $Z^{n,j}$ be defined by (\ref{fnls3_6}). Assume that $\mu$ is large enough and $k$ is sufficiently small satisfying (\ref{fnlsX2}). Then
\begin{eqnarray}
\max_{j,n}|Z^{n,j}|_{\infty}\leq C_{Z},\label{fnls_B0}
\end{eqnarray}
for some constant $C_{Z}$.
\end{lemma}
\begin{proof}
Let $\zeta^{n,1}\in S_{N}$ be satisfying
\begin{eqnarray*}
u^{N}(\tau_{n}^{1})=u^{N}+kb_{1}F\left(\frac{u^{N}(\tau_{n}^{1})+u^{N}}{2}\right)+\zeta^{n,1},
\end{eqnarray*}
where in the sequel $u^{N}=u^{N}(t_{n}), u_{t}^{N}=u_{t}^{N}(t_{n})$, etc; that is, \cite{DD2021}
\begin{eqnarray*}
\zeta^{n,1}&=&u^{N}(\tau_{n}^{1})-u^{N}-kb_{1}\left(-i(-\partial_{xx})^{s}\left(\frac{u^{N}(\tau_{n}^{1})+u^{N}}{2}\right)\right.\nonumber\\
&&\left.+iP_{N}f\left(\frac{u^{N}(\tau_{n}^{1})+u^{N}}{2}\right)\right)\\
&=&u^{N}(\tau_{n}^{1})-u^{N}-kb_{1}\left(\left(\frac{u^{N}_{t}(\tau_{n}^{1})+u_{t}^{N}}{2}\right)\right.\\
&&\left.-iP_{N}\left(\frac{f(u^{N}(\tau_{n}^{1}))+f(u^{N})}{2}-f\left(\frac{u^{N}(\tau_{n}^{1})+u^{N}}{2}\right)\right)\right)=\omega_{1}^{n}-\omega_{2}^{n},
\end{eqnarray*}
where
\begin{eqnarray*}
\omega_{1}^{n}&=&u^{N}(\tau_{n}^{1})-u^{N}-kb_{1}\left(\left(\frac{u^{N}_{t}(\tau_{n}^{1})+u_{t}^{N}}{2}\right)\right)\\
&=&kb_{1}u_{t}^{N}+\frac{k^{2}b_{1}^{2}}{2}u_{tt}^{N}+\rho_{3}-kb_{1}\left(u_{t}^{N}+\frac{kb_{1}}{2}u_{tt}^{N}+\rho_{2}\right),
\end{eqnarray*}
with
\begin{eqnarray*}
||\rho_{2}||\lesssim k^{2}\max_{t}||\partial_{t}^{2}u^{N}||,\;
||\rho_{3}||\lesssim k^{3}\max_{t}||\partial_{t}^{3}u^{N}||.
\end{eqnarray*}
Using Proposition \ref{propos22} with $\mu$ large enough, then
\begin{eqnarray}
||\omega_{1}^{n}||\lesssim k^{3}.\label{fnls_B1}
\end{eqnarray}
On the other hand, $\omega_{2}^{n}=kb_{1}(\rho^{n}-\sigma^{n})$, where
\begin{eqnarray*}
\rho^{n}&=&iP_{N}\left(\frac{f(u^{N}(\tau_{n}^{1}))+f(u^{N})}{2}\right)=\frac{1}{2}\left(u_{t}^{N}(\tau_{n}^{1})+u_{t}^{N}\right)\\
&&+\frac{i}{2}(-\partial_{xx}^{s}\left(u^{N}(\tau_{n}^{1})+u^{N}\right).
\end{eqnarray*}
Let $s_{n}^{1}=\frac{1}{2}(t_{n}+\tau_{n}^{1})=t_{n}+\frac{kb_{1}}{2}$. Expanding $u_{t}^{N}(\tau_{n}^{1})$ and $u_{t}^{N}$ about $s_{n}^{1}$ and using Proposition \ref{propos22} we have, for $\mu$ large enough and $k$ sufficiently small, that
\begin{eqnarray*}
\rho^{n}=u_{t}^{N}(s_{n}^{1})+i(-\partial_{xx})^{s}u^{N}(s_{n}^{1})+\widetilde{\rho^{n}},
\end{eqnarray*}
with
\begin{eqnarray}
||\widetilde{\rho^{n}}||\lesssim k^{2}.\label{fnls_B2}
\end{eqnarray}
Let
\begin{eqnarray*}
\eta^{n,1}=\frac{1}{2}\left(u(\tau_{n}^{1})+u^{N}\right)-u^{N}(s_{n}^{1}).
\end{eqnarray*}
Taylor's expansions of  $u^{N}(\tau_{n}^{1})$ and $u^{N}$ about $s_{n}^{1}$ and Proposition \ref{propos22} lead to
\begin{eqnarray}
||\eta^{n,1}||\lesssim k^{2},\label{fnls_B3}
\end{eqnarray}
for $\mu$ large enough. Then
\begin{eqnarray*}
\sigma^{n}&=&iP_{N}f\left(\frac{u^{N}(\tau_{n}^{1})+u^{N}}{2}\right)=iP_{N}f\left(u^{N}(s_{n}^{1})+\eta^{n,1}\right)\\
&=&iP_{N}\left(f\left(u^{N}(s_{n}^{1})+\eta^{n,1}\right)-f\left(u^{N}(s_{n}^{1})\right)\right)+iP_{N}f\left(u^{N}(s_{n}^{1})\right)\\
&=&\widetilde{\sigma^{n}}+iP_{N}f\left(u^{N}(s_{n}^{1})\right).
\end{eqnarray*}
Now, using Proposition \ref{propos11}, Proposition \ref{propos22}, and (\ref{fnls_B3}), for $\mu$ large enough, we have
\begin{eqnarray}
||\widetilde{\sigma^{n}}||\leq C(R)||\eta^{n,1}||\lesssim k^{2}.\label{fnls_B4}
\end{eqnarray}
Therefore
\begin{eqnarray*}
\omega_{2}^{n}&=&kb_{1}\left(u_{t}^{N}(s_{n}^{1})+i(-\partial_{xx})^{s}u^{N}(s_{n}^{1})-iP_{N}f\left(u^{N}(s_{n}^{1})\right)+\widetilde{\rho^{n}}-\widetilde{\sigma^{n}}\right)\\
&=&kb_{1}(\widetilde{\rho^{n}}-\widetilde{\sigma^{n}}),
\end{eqnarray*}
and from (\ref{fnls_B2}), (\ref{fnls_B4}), it holds that
\begin{eqnarray*}
||\omega_{2}^{n}||\lesssim k^{3},
\end{eqnarray*}
which, along with (\ref{fnls_B1}), implies
\begin{eqnarray}
||\zeta^{n,1}||\lesssim k^{3},\label{fnls_B5}
\end{eqnarray}
for all $n$, $\mu$ large enough and $k$ sufficiently small. Note now that
\begin{eqnarray}
Z^{n,1}-u^{N}(\tau_{n}^{1})&=&kb_{1}\left(F\left(\frac{Z^{n,1}+u^{N}}{2}\right)-F\left(\frac{u(\tau_{n}^{1})+u^{N}}{2}\right)\right)-\zeta^{n,1}\nonumber\\
&=&kb_{1}\left(-i(-\partial_{xx})^{s}\left(\frac{Z^{n,1}-u^{N}(\tau_{n}^{1})}{2}\right)\right.\label{fnls_B6}\\
&&\left.+iP_{N}\left(f\left(\frac{Z^{n,1}+u^{N}}{2}\right)-f\left(\frac{u(\tau_{n}^{1})+u^{N}}{2}\right)\right)\right)-\zeta^{n,1}\nonumber.
\end{eqnarray}
We consider the inner product of (\ref{fnls_B6}) with $Z^{n,1}-u^{N}(\tau_{n}^{1})$ and take the real part. Using that 
$${\rm Re}\left(-i(-\partial_{xx})^{s}\left(\frac{Z^{n,1}-u^{N}(\tau{n}^{1})}{2}\right),Z^{n,1}-u^{N}(\tau_{n}^{1})\right)=0,$$
(\ref{fnlsL2b}), (\ref{Inverse}), and (\ref{fnls3_2a}), then we have
\begin{eqnarray*}
||Z^{n,1}-u^{N}(\tau_{n}^{1})||^{2}\leq  3R^{2}kN|b_{1}|||Z^{n,1}-u^{N}(\tau_{n}^{1})||^{2}+||Z^{n,1}-u^{N}(\tau_{n}^{1})||||\zeta^{n,1}||.
\end{eqnarray*}
From (\ref{fnls_B5}) and the hypothesis (\ref{fnlsX2}), then it holds that
\begin{eqnarray}
||Z^{n,1}-u^{N}(\tau_{n}^{1})||\lesssim k^{3}.\label{fnls_B7}
\end{eqnarray}
Therefore, (\ref{fnls_B7}) and Proposition \ref{propos22} with $\mu$ large enough
prove (\ref{fnls_B0}) when $j=1$.

For the case $j=2$ one proceeds in a similar way: let $\zeta^{n,2}\in S_{N}$ be defined by
\begin{eqnarray*}
u^{N}(\tau_{n}^{2})=u^{N}(\tau_{n}^{1})+kb_{2}F\left(\frac{u(\tau_{n}^{1})+u^{N}(\tau_{n}^{2})}{2}\right)+\zeta^{n,2}.
\end{eqnarray*}
The same arguments as before can be used to prove that
\begin{eqnarray}
\max_{n}||\zeta^{n,2}||\lesssim k^{3}.\label{fnls_B8}
\end{eqnarray}
Now we write
\begin{eqnarray}
Z^{n,2}-u^{N}(\tau_{n}^{2})&=&Z^{n,1}-u^{N}(\tau_{n}^{1})\label{fnls_B9}\\
&&+kb_{2}\left(F\left(\frac{Z^{n,2}+Z^{n,1}}{2}\right)-F\left(\frac{u(\tau_{n}^{2})+u(\tau_{n}^{1})}{2}\right)\right)-\zeta^{n,2}.\nonumber
\end{eqnarray}
We define, \cite{DD2021}, $\chi_{j}\in S_{N}, 1\leq j\leq 4$, as
\begin{eqnarray*}
&&\chi_{1}=Z^{n,1}-u^{N}(\tau_{n}^{1}),\; \chi_{2}=Z^{n,2}-u^{N}(\tau_{n}^{2}),\\
&&\chi_{3}=\frac{Z^{n,1}+Z^{n,2}}{2},\; \chi_{4}=\frac{u(\tau_{n}^{2})+u(\tau_{n}^{1})}{2}.
\end{eqnarray*}
Note then that (\ref{fnls_B9}) can be written as
\begin{eqnarray}
\chi_{2}-\chi_{1}=kb_{2}(F(\chi_{3})-F(\chi_{4}))-\zeta^{n,2}.\label{fnls_B10}
\end{eqnarray}
We now take the real part of the inner product of (\ref{fnls_B10}) with $\frac{\chi_{1}+\chi_{2}}{2}=\chi_{3}-\chi_{4}$. Since
$${\rm Re}(i(-\partial_{xx})^{s}(\chi_{3}-\chi_{4}),\chi_{3}-\chi_{4})=0,$$ then, using again (\ref{fnlsL2b}), (\ref{Inverse}), and (\ref{fnls3_2a}) we have
\begin{eqnarray*}
\frac{1}{2}\left(||\chi_{2}||^{2}-||\chi_{1}||^{2}\right)&\leq & k|b_{2}|\left(3R^{2}N||\chi_{3}-\chi_{4}||^{2}+\frac{||\zeta^{n,2}||}{2}\left(||\chi_{1}||+||\chi_{2}||\right)\right)\\
&\leq & \frac{k|b_{2}|}{2}\left(\frac{3R^{2}N}{2}(||\chi_{1}||+||\chi_{2}||)^{2}+{||\zeta^{n,2}||}\left(||\chi_{1}||+||\chi_{2}||\right)\right),
\end{eqnarray*}
that is
\begin{eqnarray*}
||\chi_{2}||-||\chi_{1}||\leq {k|b_{2}|}\left(\frac{3R^{2}N}{2}\left(||\chi_{1}||+||\chi_{2}||\right)+||\zeta^{n,2}||\right).
\end{eqnarray*}
Now 
(\ref{fnls_B7}), (\ref{fnls_B8}), and (\ref{fnlsX2}) imply, for $\mu$ large enough 
\begin{eqnarray*}
||\chi_{2}||\leq ||\chi_{1}||+Ck^{3}\lesssim k^{3},
\end{eqnarray*}
and the same arguments as in the case $j=1$ are used to obtain
\begin{eqnarray}
||Z^{n,2}||_{\infty}\leq C,\label{fnls_B11}
\end{eqnarray}
for all $n$ and some constant $C$, under the hypothesis (\ref{fnlsX2}). 

Finally, it is not hard to see that the steps to derive (\ref{fnls_B11}) can be applied, in an analogous way, to obtain the bounds 
$$||Z^{n,j}-u^{N}(\tau_{n}^{j})||\lesssim k^{3},$$ and (\ref{fnls_B0}) for $j$ from those corresponding to the previous step $j-1$, implying then (\ref{fnls_B0}) for all $1\leq j\leq q$.
\end{proof}
Lemma \ref{lemma33a} can be extended to the norm in $H^{s}, 1/2<s\leq 1$, in the following sense.
\begin{lemma}
\label{lemma33b}
Let $Z^{n,j}$ be defined by (\ref{fnls3_6}), $1/2<s\leq 1$. Assume that $\mu$ is large enough and $k$ is sufficiently small satisfying (\ref{fnlsX2}). Then
\begin{eqnarray}
\max_{j,n}||Z^{n,j}||_{s}\leq C^{\prime}_{Z},\label{fnlsX3}
\end{eqnarray}
for some constant $C^{\prime}_{Z}$.
\end{lemma}
\begin{proof}
Using Propositions \ref{propos12} and \ref{propos22}, we check that (\ref{fnls_B1})-(\ref{fnls_B5}) are valid in the norm of $H^{s}$. Now we take the real part of the inner product of (\ref{fnls_B6}) with $(1+(-\partial_{xx})^{s})Z^{n,1}-u^{N}(\tau_{n}^{1})$, use that
$${\rm Re}\left(-i(-\partial_{xx})^{s}\left(\frac{Z^{n,1}-u^{N}(\tau{n}^{1})}{2}\right),(1+(-\partial_{xx})^{s})(Z^{n,1}-u^{N}(\tau_{n}^{1}))\right)=0,$$
then, integrating by parts, from Proposition \ref{propos12}, (\ref{Inverse}), and (\ref{fnls3_2a}), we have
\begin{eqnarray*}
||Z^{n,1}-u^{N}(\tau_{n}^{1})||_{s}^{2}\leq  3R^{2}kN^{s}|b_{1}|||Z^{n,1}-u^{N}(\tau_{n}^{1})||_{s}^{2}+||Z^{n,1}-u^{N}(\tau_{n}^{1})||_{s}||\zeta^{n,1}||_{s}.
\end{eqnarray*}
From (\ref{fnlsX2}) and since $s\leq 1$, it holds that
\begin{eqnarray*}
||Z^{n,1}-u^{N}(\tau_{n}^{1})||_{s}\lesssim ||\zeta^{n,1}||_{s}\lesssim k^{3}.
\end{eqnarray*}
From Proposition \ref{propos22} and (\ref{fnlsX2})
\begin{eqnarray}
||Z^{n,1}||_{s}\leq ||Z^{n,1}-u^{N}(\tau_{n}^{1})||_{s}+||u^{N}(\tau_{n}^{1})||_{s}\leq C,\label{fnlsX4}
\end{eqnarray}
for $\mu$ large enough. Using (\ref{fnlsX4}), the same arguments can be applied to obtain the corresponding bound for $||Z^{n,2}||_{s}$ and, recursively, (\ref{fnlsX3}) follows.
\end{proof}
\begin{lemma}
\label{lemma33}
 Assume that $u$ belongs to $H^{\mu}$ for $\mu$ large enough. Let 
$$R=\max_{0\leq t\leq T}|u^{N}(t)|_{\infty}.$$ 
Assume that $k$ is small enough so that (\ref{fnlsX2}) holds and
\begin{eqnarray}
\max_{0\leq n\leq M-1}||\Theta^{n}||\lesssim k^{\alpha+1},\label{fnls3_7}
\end{eqnarray}
for some integer $\alpha>0$, where $\Theta^{n}$ is given by (\ref{fnls3_6a}). Let $U^{n}$ be the solution of (\ref{fnls3_1}) (cf. Lemma \ref{lemma32}) and $\epsilon^{n}=Z^{n}-U^{n}$. Then
\begin{eqnarray}
\max_{0\leq n\leq M-1}||\epsilon^{n}||\lesssim k^{\alpha},\label{fnls3_8}
\end{eqnarray}
and, for $k$ small enough
\begin{eqnarray}
|Y^{n,j}|_{\infty}\leq 2\max\{R,C_{Z}\},\, \forall n,\; 0\leq j\leq q,\label{fnls3_25b}
\end{eqnarray}
where $C_{Z}$ is given by (\ref{fnls_B0}).
\end{lemma}
\begin{proof}
Let $\epsilon^{n,1}=Z^{n,1}-Y^{n,1}$. Then
\begin{eqnarray}
\epsilon^{n,1}-\epsilon^{n}=kb_{1}\left(F\left(\frac{Z^{n,1}+Z^{n}}{2}\right)-F\left(\frac{Z^{n,1}+Z^{n}}{2}-\left(\frac{\epsilon^{n,1}+\epsilon^{n}}{2}\right)\right)\right).\label{fnls3_8a}
\end{eqnarray}
In order to bound $||\epsilon^{n,1}||$ in terms of $||\epsilon^{n}||$, we make use of the second part of Lemma \ref{lemma32} and a bootstrap argument, by assuming first that
\begin{eqnarray}
|Y^{n,j}|_{\infty}\leq R_{1},\; 0\leq j\leq q,\label{fnls3_26a}
\end{eqnarray}
for some $R_{1}>2\max\{R,C_{Z}\}$.
We take the real part of the inner product of (\ref{fnls3_8a}) with $\frac{\epsilon^{n,1}+\epsilon^{n}}{2}$, use Proposition \ref{propos11} and (\ref{fnls3_26a}) to have, for some constant $C$, independent of $N$,
\begin{eqnarray*}
\frac{1}{2}\left(||\epsilon^{n,1}||^{2}-||\epsilon^{n}||^{2}\right)&=&-kb_{1}\left(f\left(\frac{Z^{n,1}+Z^{n}}{2}\right)\right.\\
&&\left.-f\left(\frac{Z^{n,1}+Z^{n}}{2}-\left(\frac{\epsilon^{n,1}+\epsilon^{n}}{2}\right)\right), \frac{\epsilon^{n,1}+\epsilon^{n}}{2}\right)\\
&\leq &k|b_{1}|C\left|\left|\frac{\epsilon^{n,1}+\epsilon^{n}}{2}\right|\right|^{2}.
\end{eqnarray*}
Therefore
\begin{eqnarray*}
||\epsilon^{n,1}||\leq \frac{1+\frac{k|b_{1}|}{2}C}{1-\frac{k|b_{1}|}{2}C}||\epsilon^{n}||,
\end{eqnarray*}
and, from the hypothesis on $k$, there is a constant $C_{1}>0$ such that
\begin{eqnarray}
||\epsilon^{n,1}||\leq (1+C_{1}k)||\epsilon^{n}||.\label{fnls3_8b}
\end{eqnarray}
The same argument can be used to prove that for $2\leq j\leq q$
\begin{eqnarray*}
||\epsilon^{n,j}||\leq \frac{1+\frac{k|b_{j}|}{2}C}{1-\frac{k|b_{j}|}{2}C}||\epsilon^{n,j-1}||,
\end{eqnarray*}
and consequently there is $C_{j}>0$ such that
\begin{eqnarray}
||\epsilon^{n,j}||\leq (1+C_{j}k)||\epsilon^{n,j-1}||, \;2\leq j\leq q.\label{fnls3_8c}
\end{eqnarray}
Therefore, (\ref{fnls3_8b}), (\ref{fnls3_8c}) imply that, if $k$ is sufficiently small, there there is a constant $C^{*}>0$ such that
\begin{eqnarray}
||\epsilon^{n,i}||\leq (1+C^{*}k)||\epsilon^{n}||.\label{fnls3_8d}
\end{eqnarray}
Note finally that $$\epsilon^{n+1}=Z^{n+1}-U^{n+1}=Z^{n+1}-Y^{n,s}=Z^{n,s}-Y^{n,s}+\Theta^{n}=\epsilon^{n,s}+\Theta^{n}.$$ 
Therefore, by (\ref{fnls3_8d}), 
\begin{eqnarray*}
||\epsilon^{n+1}||\leq (1+C^{*}k)||\epsilon^{n}||+||\Theta^{n}||.
\end{eqnarray*}
Since $\epsilon^{0}=0$, using (\ref{fnls3_7}) and the discrete Gronwall's inequality, (\ref{fnls3_8}) follows. Observe now that
\begin{eqnarray*}
|U^{n}|_{\infty}&\leq&|\epsilon^{n}|_{\infty}+|Z^{n}|_{\infty}\lesssim N^{1/2}||\epsilon^{n}||+|Z^{n}|_{\infty}\\
&\lesssim&N^{1/2}k^{\alpha}+R,\\
|Y^{n,j}|_{\infty}&\leq&|\epsilon^{n,j}|_{\infty}+|Z^{n,j}|_{\infty}\lesssim N^{1/2}||\epsilon^{n,j}||+|Z^{n,j}|_{\infty}\\
&\lesssim&N^{1/2}(1+C^{*}k)||\epsilon^{n}||+C_{Z}\lesssim N^{1/2}k^{\alpha}+C_{Z},
\end{eqnarray*}
for $1\leq j\leq q$, which implies (\ref{fnls3_25b}) for $k$ small enough, and the bootstrap argument is complete.
\end{proof}
\begin{remark}
Note that the hypothesis (\ref{fnlsX2}) in Lemma \ref{lemma33} can be replaced by a condition $kN^{1/2}=O(1)$ and $k$ small enough. In such case, $U^{n}$ is a solution of (\ref{fnls3_1}), since the uniqueness results given by Lemma \ref{lemma32} does not apply. 
\end{remark}
As in Lemma \ref{lemma33a}, one can prove a version of Lemma \ref{lemma33} in $H^{s}, 1/2<s\leq 1$.
\begin{lemma}
\label{lemma33c}
 Assume that $u$ belongs to $H^{\mu}$ for $\mu$ large enough. Let 
$$R=\max_{0\leq t\leq T}||u^{N}(t)||_{s}.$$ 
Assume that $k$ is small enough so that (\ref{fnlsX2}) holds and
\begin{eqnarray}
\max_{0\leq n\leq M-1}||\Theta^{n}||_{s}\lesssim k^{\alpha+1},\label{fnls3_7b}
\end{eqnarray}
for some integer $\alpha>1$. Let $U^{n}$ be the solution of (\ref{fnls3_1}) and $\epsilon^{n}=Z^{n}-U^{n}$. Then
\begin{eqnarray}
\max_{0\leq n\leq M-1}||\epsilon^{n}||_{s}\lesssim k^{\alpha},\label{fnls3_8bb}
\end{eqnarray}
and, for $k$ small enough
\begin{eqnarray}
||Y^{n,j}||_{s}\leq 2\max\{R,C^{\prime}_{Z}\},\, \forall n,\; 0\leq j\leq q,\label{fnls3_25c}
\end{eqnarray}
where $C_{Z}$ is given by (\ref{fnlsX3}).
\end{lemma}
\begin{proof}
We just outline the main differences with respect to the proof of Lemma \ref{lemma33}. Now, the bootstrap argument makes use of 
the assumption
\begin{eqnarray}
||Y^{n,j}||_{s}\leq R_{1},\; 0\leq j\leq q,\label{fnls3_26b}
\end{eqnarray}
for some $R_{1}>2\max\{R,C^{\prime}_{Z}\}$. Then we can take the real part of the inner product of  (\ref{fnls3_8a}) with $(1+(-\partial_{xx})^{s})\left(\frac{\epsilon^{n,1}+\epsilon^{n}}{2}\right)$, use Proposition \ref{propos12} and (\ref{fnls3_26b}) to have, for some constant $C$
\begin{eqnarray*}
\frac{1}{2}\left(||\epsilon^{n,1}||_{s}^{2}-||\epsilon^{n}||_{s}^{2}\right)
&\leq &k|b_{1}|C\left|\left|\frac{\epsilon^{n,1}+\epsilon^{n}}{2}\right|\right|_{s}^{2},
\end{eqnarray*}
leading to, as in (\ref{fnls3_8b})
\begin{eqnarray}
||\epsilon^{n,1}||_{s}\leq (1+C_{1}k)||\epsilon^{n}||_{s}.\label{fnls3_8cc}
\end{eqnarray}
From (\ref{fnls3_8cc}), the same arguments as those used in the proof of Lemma \ref{lemma33} applies with the norm in $H^{s}$ to obtain (\ref{fnls3_8bb}) and (\ref{fnls3_25c}).
\end{proof}

\begin{theorem}
\label{th33}
Under the assumptions of Lemma \ref{lemma33}, the scheme (\ref{fnls3_1}) has for all $n$ a unique solution such that
\begin{eqnarray}
\max_{n}||U^{n}-u(t_{n})||\lesssim k^{\alpha}+N^{-\mu}.\label{fnls3_9}
\end{eqnarray}
\end{theorem}
\begin{proof}
Note that, due to Lemma \ref{lemma32} and the hypothesis on $k$, uniqueness of $U^{n}$ is ensured. On the other hand
\begin{eqnarray*}
||U^{n}-u(t_{n})||&\leq & ||U^{n}-u^{N}(t_{n})||+||u^{N}(t_{n})-u(t_{n})||\\
&=&||\epsilon^{n}||+||u^{N}(t_{n})-u(t_{n})||,
\end{eqnarray*}
and from (\ref{fnls3_8}) and (\ref{fnls2_10bb}), (\ref{fnls3_9}) holds.
\end{proof}
Using the error estimate (\ref{fnls2_10b}), instead of (\ref{fnls2_10bb}), the corresponding convergence result in $H^{s}, 1/2<s\leq 1$, would be as follows.
\begin{theorem}
\label{th33a}
Under the assumptions of Lemma \ref{lemma33c}, the scheme (\ref{fnls3_1}) has for all $n$ a unique solution such that
\begin{eqnarray}
\max_{n}||U^{n}-u(t_{n})||_{s}\lesssim k^{\alpha}+N^{s-\mu}.\label{fnls3_30b}
\end{eqnarray}
\end{theorem}
\begin{remark}
Note that, from (\ref{fnls2_10bb}) and (\ref{fnls3_30b})
\begin{eqnarray*}
\max_{0\leq t\leq T}||u^{N}(t)||_{s}&\lesssim&N^{s-\mu}+\max_{0\leq t\leq T}||u(t)||_{s},\\
\max_{0\leq n\leq M}||U^{n}||_{s}&\lesssim&k^{\alpha}+N^{s-\mu}+\max_{0\leq t\leq T}||u(t)||_{s},
\end{eqnarray*}
for $1/2<s\leq 1$. In the case of the $L^{2}$ norm, the same estimates hold with $s=0$.
\end{remark}

\begin{remark}
\label{remark1}
Each internal stage in (\ref{fnls3_1}) can be solved by iteration separately: given $Y^{n,j-1}$ then $Y^{n,j}$ can be obtained from the resolution of the fixed point system for $X^{*}\in S_{N}$
\begin{eqnarray}
X^{*}=Y^{n,j-1}+\frac{kb_{j}}{2}F(X^{*}),\label{fnls_B12}
\end{eqnarray}
and then taking 
\begin{eqnarray}
Y^{n,j}=2X^{*}-Y^{n,j-1}.\label{fnls_B12a}
\end{eqnarray}
The implementation of section \ref{sec4} makes use of the following fixed point-type iterative method for (\ref{fnls_B12})
\begin{eqnarray}
X_{0}&=&Y^{n,j-1},\nonumber\\
\left(I+i\frac{kb_{j}}{2}(-\partial_{xx})^{s}\right)X_{\nu+1}&=&Y^{n,j-1}+i\frac{kb_{j}}{2}f(X_{\nu}),\label{fnls_B13}
\end{eqnarray}
for $\nu=0,1,\ldots$, and $X_{\nu}\in S_{N}$. Note from (\ref{fnls_B12a}) and Lemma \ref{lemma33} that
$|X^{*}|_{\infty}\leq R^{*}$, for some $R^{*}>0$. From (\ref{fnls_B12}) and (\ref{fnls_B13}) we have
\begin{eqnarray}
\left(I+i\frac{kb_{j}}{2}(-\partial_{xx})^{s}\right)\left(X^{*}-X_{\nu+1}\right)=i\frac{kb_{j}}{2}\left(f(X^{*})-f(X_{\nu})\right).\label{fnls_B14}
\end{eqnarray}
We first consider the case $\nu=0$. If we take the real part of the inner product of (\ref{fnls_B14}) with $X^{*}-X_{0}$ and using Proposition \ref{propos11}, there is a constant $C$ such that
\begin{eqnarray*}
||X^{*}-X_{1}||\leq  \frac{k|b_{1}|}{2}C||X^{*}-X_{0}||\leq \frac{k|b_{1}|}{\sqrt{2}}C\sqrt{L}|X^{*}-X_{0}|_{\infty}.
\end{eqnarray*}
Therefore
\begin{eqnarray*}
|X_{1}|_{\infty}&\leq & |X^{*}-X_{1}|_{\infty}+|X^{*}|_{\infty}\lesssim N^{1/2}||X^{*}-X_{1}||+|X^{*}|_{\infty}\\
&\lesssim & kN^{1/2}+R^{*},
\end{eqnarray*}
and from (\ref{fnlsX2}) we have $|X_{1}|_{\infty}\leq 1+R^{*}$.
Now, using induction on $\nu$, if we assume $|X_{\nu}|_{\infty}\leq 1+R^{*}$, taking  the real part of the inner product of (\ref{fnls_B14}) with $X^{*}-X_{\nu+1}$ and using Proposition \ref{propos11}, there is a constant $C$ such that
\begin{eqnarray}
||X^{*}-X_{\nu+1}||\leq  \frac{k|b_{j}|}{2}C||X^{*}-X_{\nu}||\leq \frac{k|b_{j}|}{\sqrt{2}}C\sqrt{L}|X^{*}-X_{\nu}|_{\infty}.\label{fnls_B16}
\end{eqnarray}
Therefore
\begin{eqnarray*}
|X_{\nu+1}|_{\infty}&\leq & |X^{*}-X_{\nu+1}|_{\infty}+|X^{*}|_{\infty}\lesssim N^{1/2}||X^{*}-X_{\nu+1}||+|X^{*}|_{\infty}\\
&\lesssim & kN^{1/2}+R^{*},
\end{eqnarray*}
from (\ref{fnls_B16}). This implies, under the hypothesis (\ref{fnlsX2}), that $|X_{\nu}|_{\infty}\leq 1+R^{*}$, for all $\nu$ and, from the first inequality in (\ref{fnls_B16}) and $k$ satisfying (\ref{fnlsX2}), it holds that
\begin{eqnarray*}
||X^{*}-X_{\nu}||\rightarrow 0,\; ||X^{*}-X_{\nu}||\rightarrow 0,
\end{eqnarray*}
as $\nu\rightarrow\infty$.
\end{remark}

\section{Numerical experiments}
\label{sec4}
We now show some numerical experiments to validate the performance of the fully discrete method introduced above. To this end, the spectral semidiscretization is formulated as a collocation method, where the numerical solution $u_{N}:[0,T]\rightarrow S_{N}$ satisfies (\ref{fnls2_0}) at a uniform grid $x_{j}=-L+jh, j=0,\ldots,N-1$, of collocation points, with $h=2L/N$, and is represented by its nodal values at the $x_{j}$ as a vector
$$U_{N}(t)=(u_{N}(x_{0},t),\ldots,u_{N}(x_{N-1},t))^{T},$$ satisfying
\begin{eqnarray}
i\frac{d}{dt}U_{N}-(-D_{N}^{2})^{s}U_{N}+f(U_{N})=0,\label{fnls4_1}
\end{eqnarray}
where $D_{N}$ represents the $N\times N$ Fourier pseudospectral differentiation matrix in $(-L,L)$. The system (\ref{fnls4_1}) is implemented in terms of the Fourier representation using FFT techniques and the corresponding computation of $D_{N}$.
For the cubic nonlinearity $f(u)=|u|^{2}u$, the corresponding term in (\ref{fnls4_1}) involves Hadamard products. As in \cite{Cano2006} for the nonfractional case ($s=1$), (\ref{fnls4_1}) admits a canonical Hamiltonian structure with Hamiltonian
\begin{eqnarray}
H(U_{N})=\frac{1}{2}\left(\langle W_{N},A_{N}W_{N}\rangle_{N}+\langle V_{N},A_{N}V_{N}\rangle_{N}-V(V_{N},W_{N})\right),\label{fnls3cc}
\end{eqnarray}
where $U_{N}=V_{N}+iW_{N}$, $A_{N}=(-D_{N}^{2})^{s}$, $V$ is given by (\ref{fnls11a}), and $\langle\cdot,\cdot\rangle_{N}$ denotes the Euclidean inner product in $\mathbb{R}^{N}$. In addition, since $A_{N}$ is symetric,
\begin{eqnarray}
I_{1}(U_{N})=\frac{1}{2}\left(\langle V_{N},V_{N}\rangle_{N}+\langle W_{N},W_{N}\rangle_{N}\right),\label{fnls3aa}
\end{eqnarray}
is an additional conserved quantity of (\ref{fnls4_1}), as well as, \cite{Cano2006}
\begin{eqnarray}
I_{2}(U_{N})=\frac{1}{2}\left(\langle V_{N},D_{N}W_{N}\rangle_{N}-\langle W_{N},D_{N}V_{N}\rangle_{N}\right).\label{fnls3bb}
\end{eqnarray}
All the experiments were made with the full discretization corresponding to the $4$th-order method (\ref{44}) with $q=3$ (cf. Appendix \ref{appA}). The aim of the first ones is to check the accuracy of the approximation. To this end, we consider the nonfractional, classical NLS equation (\ref{fnls11}) with $s=1$, and its solitary wave (soliton) solutions
\begin{eqnarray}
\psi(x,t,\lambda_{0}^{1},\lambda_{0}^{2},x_{0},\theta_{0})&=&G_{(t\lambda_{0}^{1},t\lambda_{0}^{2})}(\varphi)\nonumber\\
&=&\rho(x-t\lambda_{0}^{2}-x_{0})e^{i(\theta(x-t\lambda_{0}^{2}-x_{0})+\theta_{0}+\lambda_{0}^{1}t)},\label{fnls22_7}
\end{eqnarray}
 where, \cite{DuranS2000}
\begin{eqnarray}
\rho(x)&=&\sqrt{2a}{\rm sech}\sqrt{a}x,\quad a=\lambda_{0}^{1}-\frac{(\lambda_{0}^{2})^{2}}{4},\nonumber\\
\theta(x)&=&\frac{\lambda_{0}^{2}}{2}x.\label{fnls_45}
\end{eqnarray}
The numerical solution at $T=100$ is first compared with the exact solution (\ref{fnls22_7}), (\ref{fnls_45}) with $\lambda_{0}^{1}=1, \lambda_{0}^{2}=0.25, x_{0}=\theta_{0}=0$. The errors in the $L^{2}$ norm of the $v$ and $w$ components, for several time-step sizes are displayed in Table \ref{texpe1}, showing the fourth order of convergence of the time discretization. (Note that, since the soliton solutions are smooth, a spectral order of convergence of the semidiscrete approximation is expected, \cite{DD2021}.)
\begin{table}[htbp]
\begin{tabular}{c|c|c|c|c|}
$\Delta t$&$v$ Error&Rate&$w$ Error&Rate\\
\hline
$2.5\times 10^{-2}$&$1.1621\times 10^{-4}$&&$1.9446\times 10^{-4}$&\\
$1.25\times 10^{-2}$&$7.2820\times 10^{-6}$&$3.9963$&$1.2187\times 10^{-5}$&$3.9961$\\
$6.25\times 10^{-3}$&$4.5630\times 10^{-7}$&$3.9963$&$7.6359\times 10^{-7}$&$3.9963$\\
$3.125\times 10^{-3}$&$2.7478\times 10^{-8}$&$4.0537$&$4.6003\times 10^{-8}$&$4.0530$\\
\hline
\end{tabular}
\caption{$L^{2}$ errors and temporal convergence rates. Solitary-wave solution (\ref{fnls_45}) with $\sigma=1, \lambda_{0}^{1}=1, \lambda_{0}^{2}=0.25$,  $T=100$, $N=4096$.\label{texpe1}}
\end{table}
The time behaviour of the error is displayed, for several time step sizes and in log-log scale, in Figure \ref{FIG_C1}. The linear growth with time observed in the figure agrees with the results shown in \cite{DuranS2000} for the simulation of a solitary wave with the IMR, and can be explained with the same arguments: The leading term of the error can be divided into two components; one is associated to the errors in the parameters of the solitary wave (amplitude, speed, and phases) and the second one, not related to them, is bounded. The preservation of the mass and momentum quantities by the method implies that, up to $O(k^{4})$ terms, amplitude and speed are preserved, while the errors in the phases grow like $O(tk^{4})$. This is the linear growth observed in Figure \ref{FIG_C1}.
\begin{figure}[htbp]
\centering
{\includegraphics[width=0.8\columnwidth]{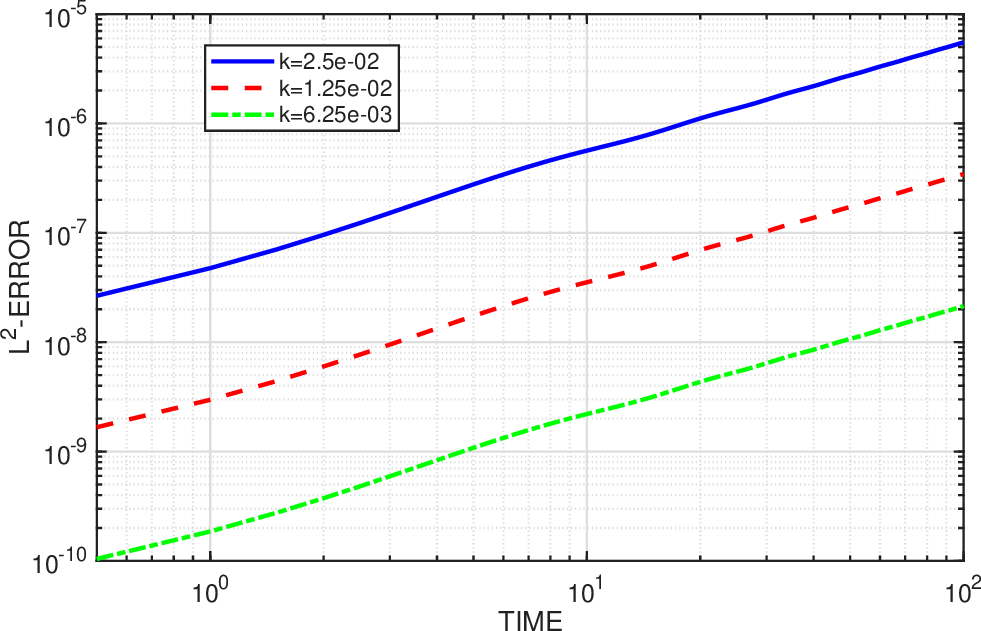}}
\caption{Time behaviour of the $L^{2}$ errrors w.r.t. the solitary-wave solution (\ref{fnls_45}) with $\lambda_{0}^{1}=1, \lambda_{0}^{2}=0.25$, $x_{0}=\theta_{0}=0$.}
\label{FIG_C1}
\end{figure}

We now measure the accuracy in the fractional case. By way of illustration, we study computationally the simulation of a solitary wave solution of (\ref{fnls11}) with $s\in (1/2,1)$, whose existence was recently proved in \cite{NRAD1}. The waves are of the form (\ref{fnls22_7}) but, unlike the nonfractional case, explicit formulas for $\rho$ and $\theta$ are not known and approximate solitary-wave profiles must be numerically generated with the iterative procedure described in \cite{NRAD1}, cf. \cite{Petv1976}. Taking one of these approximate profiles as initial condition for the full discretization, we illustrate the accuracy of the simulation with two types of experiments. In Figure \ref{FIG_C2}, the time evolution of the error of the quantities (\ref{fnls3cc})-(\ref{fnls3bb}) is displayed, in log-log scale, up to a final time $T=100$, and for $k=1.25\times 10^{-2}$. The figures show a virtual preservation of the invariants, up to the final time of simulation. The small error growth is due to the propagation of some errors in the accuracy of the first steps of the iterative procedure (\ref{fnls_B13}) of the internal stages of the fully discrete method, \cite{FSS1992}.

\begin{figure}[htbp]
\centering
{\includegraphics[width=0.8\columnwidth]{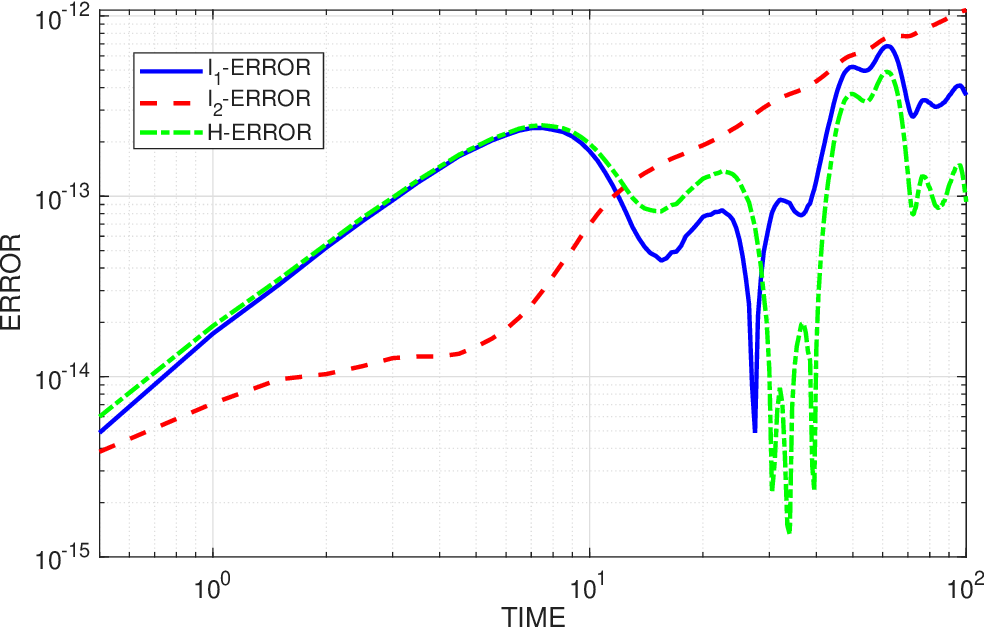}}
\caption{Time behaviour of the errors of the quantities (\ref{fnls3cc})-(\ref{fnls3cc}) w.r.t. a solitary-wave solution of (\ref{fnls11}) with $s=0.75, \lambda_{0}^{1}=1, \lambda_{0}^{2}=0.25$, with $k=1.25\times 10^{-2}$.}
\label{FIG_C2}
\end{figure}
A second group of experiments is concerned with the accuracy  of the computation of several parameters of the waves. Figure \ref{FIG_C3} shows the evolution of the errors in the amplitude and speed of the solitary waves for several step sizes. By comparison with the arguments above, concerning the growth with time of the error when approximating solitons in the nonfractional case, these results reveal that these parameters are almost conserved in the simulation, and give some confidence on the qualitative accuracy of the long-term simulation of the waves.

\begin{figure}[htbp]
\centering
\subfigure
{\includegraphics[width=6.2cm]{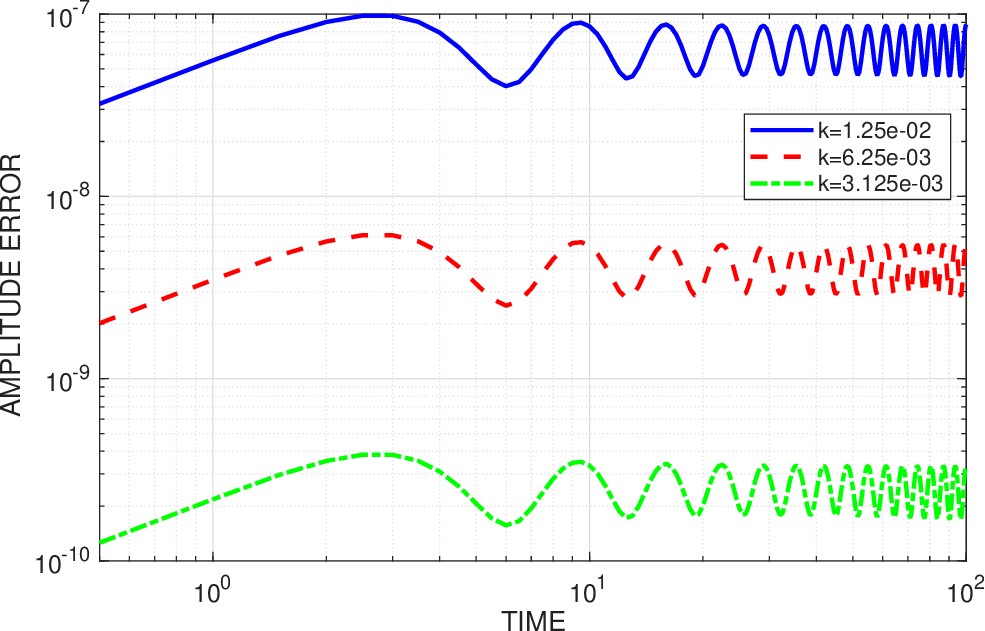}}
\subfigure
{\includegraphics[width=6.2cm]{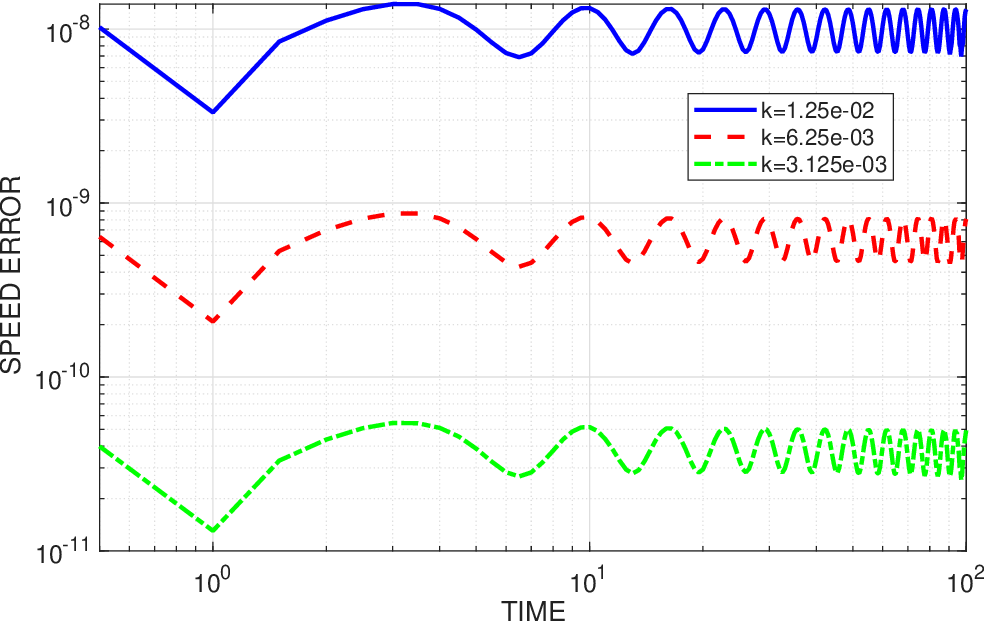}}
\caption{Time behaviour of the errors in (a) amplitude; (b) speed w.r.t. a solitary-wave solution of (\ref{fnls11}) with $s=0.75, \lambda_{0}^{1}=1, \lambda_{0}^{2}=0.25$.}
\label{FIG_C3}
\end{figure}



\section*{Acknowledgments}
This research has been supported by Ministerio de Ciencia e Innovaci\'on project PID2023-147073NB-I00.

\bibliographystyle{plain}

\appendix
\section{Local temporal error of the fully discrete method for $q=1,3$}
\label{appA}
The satisfaction of (\ref{fnls3_7}) depends on the particular form of the nonlinearity $f$ and may require additional hypotheses on its regularity and on the parameter $s$ of the fractional Laplacian. This will be illustrated in this appendix, by deriving the condition (\ref{fnls3_7}) for the methods (\ref{44}) with $q=1$ (the implicit midpoint rule, $\alpha=2$) and $q=3$ (with order $\alpha=4$), and then making some comments on how to proceed for a more general term $f$. The outline of the proofs is similar to that in \cite{DD2021} for the KdV equation, although some arguments must be different, due to the different form of the fNLS equation. We will focus, for simplicity, on the $L^{2}$ norm, but similar arguments can be used to derive (\ref{fnls3_7b}).

In what follows we will make use of several properties, listed below.
\begin{itemize}
\item[(a1)] We recall that the $Z^{n,i}$ defined by (\ref{fnls3_6}) exist and satisfy Lemma \ref{lemma33} and $||Z^{n,i}||=||u^{N}(t_{n})||=||u^{N}(0)||$, for all $n,i$.
\item[(a2)] As mentioned above, we may have to extend well-posedness of (\ref{fnls2_0}) and the definition of the semidiscrete solution (\ref{fnls2_2}) for intervals of the form $[-k,T+k]$, where (\ref{fnls2_10b}) and (\ref{fnls2_13}) hold.
\end{itemize}
In the sequel we will use both forms (as complex or real-valued vectors) of the functions involved, as well as the corresponding complex, $(\cdot,\cdot)$, or real, $\langle\cdot,\cdot\rangle={\rm Re}(\cdot,\cdot)$, inner products. Additional notation will involve expressions like $u^{N}, u_{t}^{N}$, etc, that, as above, will denote the evaluation $u^{N}(t_{n}), u_{t}^{N}(t_{n})$, etc, of the corresponding function at time $t_{n}$.
\subsection{The case $f(u)=|u|^{2}u$ and the implicit midpoint rule $q=1$}
\label{secA1}
\begin{proposition}
\label{proposA0}
Let $Z^{n,1},\Theta^{n}$ be defined by (\ref{fnls3_6}), (\ref{fnls3_6a}) for the RK Composition method (\ref{44}) with $q=1$ (Implicit midpoint rule). Let $u$ be the solution of (\ref{fnls2_0}) and assume that $u\in H^{\mu}$ for $0\leq t\leq T$ and $\mu$ large enough. Then, for $k$ small enough, there is a constant $C$, independent of $k$ and $N$, such that
\begin{eqnarray}
\max_{0\leq n\leq M-1}||\Theta^{n}||\leq Ck^{3}.\label{fnls_AA0}
\end{eqnarray}
\end{proposition}
\begin{proof}
We write
\begin{eqnarray}
Z^{n,1}=u(t_{n+1})+e^{n,1},\label{fnls_AA1}
\end{eqnarray}
for some $e^{n,1}\in S_{N}$, and from (\ref{fnls3_6a}), we need to prove that
\begin{eqnarray}
||e^{n,1}||\lesssim k^{3}.\label{fnls_AA2}
\end{eqnarray}
Taylor's expansion of (\ref{fnls_AA1}) about $t_{n}$ leads to
\begin{eqnarray}
Z^{n,1}=u^{N}+ku_{t}^{N}+\frac{k^{2}}{2}u_{tt}^{N}+\rho_{1}+e^{n,1},\label{fnls_AA3}
\end{eqnarray}
where
\begin{eqnarray}
||\rho_{1}||_{j}\lesssim k^{3}\max_{t}||\partial_{t}^{3}u^{N}||_{j},\; j\geq 0.\label{fnls_AA4}
\end{eqnarray}
On the other hand, using (\ref{fnls3_6}) and (\ref{fnls_AA3}), we have
\begin{eqnarray}
Z^{n,1}&=&u^{N}+ik\left(-(-\partial_{xx})^{s}\frac{1}{2}\left(2u^{N}++ku_{t}^{N}+\frac{k^{2}}{2}u_{tt}^{N}+\rho_{1}+e^{n,1}\right)\right)\nonumber\\
&&+\frac{ik}{8}P_{N}\left(|u^{N}+u^{N}(t_{n+1})|^{2}(u^{N}+u^{N}(t_{n+1}))+\mathcal{A}(e^{n,1})\right),\label{fnls_AA5}
\end{eqnarray}
where
\begin{eqnarray}
\mathcal{A}(e^{n,1})&=&|u^{N}+u^{N}(t_{n+1})|^{2}e^{n,1}+\mathcal{M}_{1}(u^{N}+u^{N}(t_{n+1})+e^{n,1}),\nonumber\\
\mathcal{M}_{1}&=&|e^{n,1}|^{2}+2{\rm Re}\left((u^{N}+u^{N}(t_{n+1}))\overline{e^{n,1}}\right)\nonumber\\
&=&{\rm Re}\left((u^{N}+u^{N}(t_{n+1}))\overline{e^{n,1}}\right)+
{\rm Re}\left((u^{N}+Z^{n,1})\overline{e^{n,1}}\right).\label{fnls_AA6}
\end{eqnarray}
We now use the expressions (\ref{fnls_AA3}) and (\ref{fnls_AA5}) and equate in powers of $k$. From (\ref{fnls2_3}) and its derivative with respect to $t$, the $O(1), O(k)$, and $O(k^{2})$ terms lead to an identity. Equating the $O(k^{3})$ and higher-order terms leads to an equation for $e^{n,1}$ of the form
\begin{eqnarray}
e^{n,1}&=&\Gamma-\frac{ik}{2}(-\partial_{xx})^{s}e^{n,1}+\frac{ik}{8}P_{N}\mathcal{A}(e^{n,1}),\label{fnls_AA7}\\
\Gamma&=&-\rho_{1}-\frac{ik}{2}(-\partial_{xx})^{s}\rho_{1}+\frac{ik}{8}P_{N}\mathcal{S},\nonumber
\end{eqnarray}
where $S$ collects the $O(k^{3})$ and higher-order terms of the expansion of $|u^{N}+u^{N}(t_{n+1})|^{2}(u^{N}+u^{N}(t_{n+1}))$ about $t_{n}$. We take the real part of the inner product of (\ref{fnls_AA7}) with $e^{n,1}$, 
\begin{eqnarray}
||e^{n,1}||^{2}&=&{\rm Re}\left(\Gamma,e^{n,1}\right)+{\rm Re}\left(\frac{ik}{2}(-\partial_{xx})^{s}e^{n,1},e^{n,1}\right)\nonumber\\
&&+{\rm Re}\left(\frac{ik}{8}P_{N}\mathcal{A}(e^{n,1}),e^{n,1}\right).\label{fnls_AA7a}
\end{eqnarray}
First note that
\begin{eqnarray*}
{\rm Re}\left(\frac{ik}{2}(-\partial_{xx})^{s}e^{n,1},e^{n,1}\right)=0.
\end{eqnarray*}
Using Proposition \ref{propos22} and (\ref{fnls_AA4}), we have, for $\mu$ large enough,
\begin{eqnarray}
{\rm Re}\left(\Gamma,e^{n,1}\right)\lesssim k^{3}||e^{n,1}||.\label{fnls_AA8}
\end{eqnarray}
Using (\ref{fnls_AA6}), Proposition \ref{propos22}, and Lemma \ref{lemma33a}, for $\mu$ large enough, and $k$ sufficiently small, it holds that
\begin{eqnarray}
{\rm Re}\left(\mathcal{A}(e^{n,1}),e^{n,1}\right)\lesssim ||e^{n,1}||^{2}.\label{fnls_AA9}
\end{eqnarray}
The application of (\ref{fnls_AA8}), (\ref{fnls_AA9}) to (\ref{fnls_AA7a}) yields, for $\mu$ large and $k$ small enough
\begin{eqnarray*}
||e^{n,1}||^{2}\lesssim k^{3}||e^{n,1}||+k||e^{n,1}||^{2},
\end{eqnarray*}
from which (\ref{fnls_AA2}), and therefore (\ref{fnls_AA0}), follow for $k$ small enough.
\end{proof}
\subsection{The case $f(u)=|u|^{2}u$ and the method with $q=3$}
\label{secA2}
The consistency of the scheme (\ref{44}) with $q=3$, for the time integration of (\ref{fnls2_0}) is given in the following result.
\begin{proposition}
\label{proposA1}
Let $Z^{n,1},\Theta^{n}$ be defined by (\ref{fnls3_6}), (\ref{fnls3_6a}) for the RK Composition method (\ref{44}) with $q=3$. Let $u$ be the solution of (\ref{fnls2_0}) and assume that $u\in H^{\mu}$ for $0\leq t\leq T$ and $\mu$ large enough. Then, for $k$ small enough, there is a constant $C$, independent of $k$ and $N$, such that
\begin{eqnarray}
\max_{0\leq n\leq M-1}||\Theta^{n}||\leq Ck^{5}.\label{fnls_A0}
\end{eqnarray}
\end{proposition}
\begin{proof}
As in \cite{DD2021}, the proof of (\ref{fnls_A0}) consists of deriving asymptotic expansions
\begin{eqnarray}
Z^{n,1}&=&u^{N}(\tau_{n}^{1})+k^{3}A_{1}+k^{4}A_{2}+e^{n,1},\label{fnls_A1}\\
Z^{n,2}&=&u^{N}(\tau_{n}^{2})+k^{3}B_{1}+k^{4}B_{2}+e^{n,2},\label{fnls_A2}
\end{eqnarray}
where $\tau_{n}^{1}=t_{n}+kb_{1}, \tau_{n}^{2}=t_{n}+k(b_{1}+b_{2})$, $u^{N}=v^{N}+iw^{N}$, for some coefficients $A_{j}, B_{j}\in S_{N}$ and remainders $e^{n,j}\in S_{N}, j=1,2$, satisfying
\begin{eqnarray}
||e^{n,j}||\lesssim k^{5}.\label{fnls_A2a}
\end{eqnarray}
The final step is proving that
\begin{eqnarray}
Z^{n,3}=Z^{n+1}+e^{n,3}=u^{N}(t_{n+1})+e^{n,3},\label{fnls_A3}
\end{eqnarray}
for some $e^{n,3}\in S_{N}$ satisfying
\begin{eqnarray}
||e^{n,3}||\lesssim k^{5}.\label{fnls_A3a}
\end{eqnarray}
We start with the derivation of (\ref{fnls_A1}). The term $Z^{n,1}$ is first expanded in the form
\begin{eqnarray}
Z^{n,1}&=&u^{N}+kb_{1}u_{t}^{N}+\frac{k^{2}b_{1}^{2}}{2}u_{tt}^{N}+\frac{k^{3}b_{1}^{3}}{6}u_{ttt}^{N}\nonumber\\
&&+\frac{k^{4}b_{1}^{4}}{24}\partial_{t}^{4}u^{N}+\rho_{1}+k^{3}A_{1}+k^{4}A_{2}+e^{n,1},\label{fnls_A5}
\end{eqnarray}
where $\rho_{1}$ is the Taylor remainder of the expansion of $u^{N}(\tau_{n}^{1})$ about $t_{n}$ with
\begin{eqnarray}
||\rho_{1}||_{j}\lesssim k^{5}\max_{t}||\partial_{t}^{5}u^{N}||_{j},\quad j\geq 0.\label{fnls_A5a}
\end{eqnarray}
From (\ref{fnls3_6}) with $i=1$, we have, using (\ref{fnls_A5})
\begin{eqnarray}
Z^{n,1}&=&u^{N}+ikb_{1}\left(-(-\partial_{xx})^{s}\frac{1}{2}\left(Z^{n,1}+u^{N}\right)\right)\nonumber\\
&&+\frac{ikb_{1}}{8}P_{N}\left(|Z^{n,1}+u^{N}|^{2}(Z^{n,1}+u^{N})\right),\nonumber\\
&=&u^{N}+ikb_{1}\left(-(-\partial_{xx})^{s}\frac{1}{2}\left(2u^{N}+kb_{1}u_{t}^{N}\right.\right.\nonumber\\
&&\left.\left.+\frac{k^{2}b_{1}^{2}}{2}u_{tt}^{N}+\frac{k^{3}b_{1}^{3}}{6}u_{ttt}^{N}+\rho_{2}+k^{3}A_{1}+k^{4}A_{2}+e^{n,1}\right)\right)\nonumber\\
&&+\frac{ikb_{1}}{8}P_{N}\left(|G_{1}|^{2}G_{1}+\mathcal{A}(e^{n,1})\right),\label{fnls_A6}
\end{eqnarray}
where
\begin{eqnarray}
||\rho_{2}||_{j}&\lesssim &k^{4}\max_{t}||\partial_{t}^{4}u^{N}||_{j},\quad j\geq 0,\label{fnls_A5b}\\
G_{1}&=&u^{N}+u^{N}(\tau_{n}^{1})+k^{3}A_{1}+k^{4}A_{2},\label{fnls_A5c}\\
\mathcal{A}(e^{n,1})&=&|G_{1}|^{2}e^{n,1}+2{\rm Re}(G_{1}\overline{e^{n,1}})+|e^{n,1}|^{2}.\label{fnls_A5d}
\end{eqnarray}
Now we equate the two expansions (\ref{fnls_A5}) and (\ref{fnls_A6}) of $Z^{n,1}$ and, taking into account (\ref{fnls_A5a})-(\ref{fnls_A5d}), we identify the coefficients in the same power of $k$ on both sides, in order to derive $A_{1}, A_{2}$ and the estimate (\ref{fnls_A2a}) for $i=1$.

\noindent{\bf 1.}
Equating the $O(k)$ terms leads to
\begin{eqnarray*}
kb_{1}u_{t}^{N}=-ikb_{1}(-\partial_{xx})^{s}u^{N}+ikb_{1}P_{N}(|u^{N}|^{2}u^{N}),
\end{eqnarray*}
which is an identity, in view of (\ref{fnls2_3}).

\noindent{\bf 2.}
Equating the $O(k^{2})$ terms leads to
\begin{eqnarray}
\frac{k^{2}b_{1}^{2}}{2}u_{tt}^{N}=-i\frac{k^{2}b_{1}^{2}}{2}(-\partial_{xx})^{s}u_{t}^{N}+i\frac{k^{2}b_{1}^{2}}{2}P_{N}(|u^{N}|^{2}u_{t}^{N}+\partial_{t}(|u^{N}|^{2})u^{N}),\label{fnls_A8}
\end{eqnarray}
which is an identity, from the differentiation of (\ref{fnls2_3}) with respect to $t$.

\noindent{\bf 3.}
Equating the $O(k^{3})$ terms leads to
\begin{eqnarray}
\frac{k^{3}b_{1}^{3}}{6}u_{ttt}^{N}+k^{3}A_{1}&=&-i\frac{k^{3}b_{1}^{3}}{4}(-\partial_{xx})^{s}u_{tt}^{N}\label{fnls_A9}\\
&&+i\frac{k^{3}b_{1}^{3}}{4}P_{N}(|u^{N}|^{2}u_{tt}^{N}+\partial_{t}(|u^{N}|^{2})u_{t}^{N}+|u_{t}^{N}|^{2}u^{N}),\nonumber
\end{eqnarray}
and equation (\ref{fnls_A9}) determines the coefficient $A_{1}$ in the expansion (\ref{fnls_A1}), which can be simplified by using (\ref{fnls_A8}).

\noindent{\bf 4.}
Equating the $O(k^{4})$ terms leads to
\begin{eqnarray*}
\frac{k^{4}b_{1}^{4}}{24}\partial_{t}^{4}u^{N}+k^{4}A_{2}&=&-i\frac{k^{4}b_{1}^{4}}{12}(-\partial_{xx})^{s}u_{ttt}^{N}\nonumber\\
&&+i\frac{kb_{1}}{8}P_{N}\left(4|u^{N}|^{2}\left(\frac{k^{3}b_{1}^{3}}{6}u_{ttt}^{N}+A_{1}k^{3}\right)\right.\nonumber\\
&&\left.+k^{3}b_{1}^{3}\left(2u_{tt}^{N}\partial_{t}|u^{N}|^{2}+|u_{t}^{N}|^{2}u_{t}^{N}\right)\right.\nonumber\\
&&\left.+k^{3}b_{1}^{3}\left(u^{N}\partial_{t}|u_{t}^{N}|^{2}+2(v^{N}v_{tt}^{N}+w^{N}w_{tt}^{N})u_{t}^{N}\right)\right.\nonumber\\
&&\left.+\frac{8}{3}k^{3}b_{1}^{3}\left(v^{N}v_{ttt}^{N}+w^{N}w_{ttt}^{N})u^{N}\right)\right),\label{fnls_A9b}
\end{eqnarray*}
from which the second coefficient $A_{2}$ is determined.

\noindent{\bf 5.}
Equating the $O(k^{5})$ and higher-order terms, and after some calculations, the system for $e^{n,1}$ is
\begin{eqnarray}
e^{n,1}+\frac{ikb_{1}}{2}(-\partial_{xx})^{s}e^{n,1}=\Gamma_{1}+\frac{ikb_{1}}{8}P_{N}\mathcal{A}(e^{n,1}),\label{fnls_A10}
\end{eqnarray}
where $\mathcal{A}(e^{n,1})$ is given by (\ref{fnls_A5d}) and
\begin{eqnarray*}
\Gamma_{1}=-\rho_{1}-\frac{ikb_{1}}{2}\left((-\partial_{xx})^{s}\left(\rho_{2}+{k^{4}}A_{2}\right)\right)+\frac{ikb_{1}}{8}P_{N}\beta,\label{fnls_A10a}
\end{eqnarray*}
where $\beta$ denotes the $k^{4}$th term of the Taylor expansion of $G_{1}$, given in (\ref{fnls_A5c}), and a remainder satisfying a bound of the form similar to (\ref{fnls_A5a}). We take the real part of the inner product of (\ref{fnls_A10}) with $e^{n,1}$ to have
\begin{eqnarray}
||e^{n,1}||^{2}+\langle ikb_{1}(-\partial_{xx})^{s}e^{n,1},e^{n,1}\rangle=\langle \Gamma_{1},e^{n,1}\rangle+\langle \frac{ikb_{1}}{8}P_{N}\mathcal{A}(e^{n,1}),e^{n,1}\rangle.\label{fnls_A12}
\end{eqnarray}
First note that
\begin{eqnarray}
\langle\frac{ik}{2}(-\partial_{xx})^{s}e^{n,1},e^{n,1}\rangle=0.\label{fnls_A11}
\end{eqnarray}
Using Proposition \ref{propos22}, (\ref{fnls_A5a}), and (\ref{fnls_A5b}), we have, for $\mu$ large enough,
\begin{eqnarray}
\langle\Gamma_{1},e^{n,1}\rangle\lesssim k^{5}||e^{n,1}||.\label{fnls_A11a}
\end{eqnarray}
We use (\ref{fnls_A1}) and (\ref{fnls_A5c}) to write (\ref{fnls_A5d}) in the form
\begin{eqnarray}
\mathcal{A}(e^{n,1})&=&|G_{1}|^{2}e^{n,1}+{\rm Re}((e^{n,1}+G_{1})\overline{e^{n,1}})+{\rm Re}(G_{1}\overline{e^{n,1}})\nonumber\\
&=&|G_{1}|^{2}e^{n,1}+{\rm Re}((Z^{n,1}+u^{N})\overline{e^{n,1}})+{\rm Re}(G_{1}\overline{e^{n,1}}).\label{fnls_A11c}
\end{eqnarray}
Then, using Proposition \ref{propos22}, and Lemma \ref{lemma33a}, for $\mu$ large enough, and $k$ sufficiently small, it holds that
\begin{eqnarray}
\langle\mathcal{A}(e^{n,1}),e^{n,1}\rangle\lesssim ||e^{n,1}||^{2}.\label{fnls_A11b}
\end{eqnarray}
The application of (\ref{fnls_A11})-(\ref{fnls_A11b}) to (\ref{fnls_A12}) yields, for $\mu$ large and $k$ small enough
\begin{eqnarray*}
||e^{n,1}||^{2}\lesssim k^{5}||e^{n,1}||+k||e^{n,1}||^{2},
\end{eqnarray*}
from which  (\ref{fnls_A2a}) follows in the case $j=1$, for $k$ small enough. Note, in addition, that this implies, using Proposition \ref{propos22} and (\ref{fnls_A11c}) in (\ref{fnls_A10}), that
\begin{eqnarray}
k||(-\partial_{xx})^{s}e^{n,1}||\lesssim k^{5}.\label{fnls_A18}
\end{eqnarray}
The estimate (\ref{fnls_A18}) will be used later on.

\medskip
We continue the proof with the derivation of (\ref{fnls_A2}). The term $Z^{n,2}$ is first expanded in the form
\begin{eqnarray}
Z^{n,2}&=&u^{N}+k(b_{1}+b_{2})u_{t}^{N}+\frac{k^{2}(b_{1}+b_{2})^{2}}{2}u_{tt}^{N}+\frac{k^{3}(b_{1}+b_{2})^{3}}{6}u_{ttt}^{N}\nonumber\\
&&+\frac{k^{4}(b_{1}+b_{2})^{4}}{24}\partial_{t}^{4}u^{N}+\rho_{3}+k^{3}B_{1}+k^{4}B_{2}+e^{n,2},\label{fnls_A20}
\end{eqnarray}
where $\rho_{3}\in S_{N}$ is the remainder of the expansion of $u^{N}(\tau_{n}^{2})$ about $t_{n}$ with
\begin{eqnarray}
||\rho_{3}||_{j}\lesssim k^{5}\max_{t}||\partial_{t}^{5}u^{N}||_{j},\quad j\geq 0.\label{fnls_A20b}
\end{eqnarray}
From (\ref{fnls3_6}) with $i=2$, we have, using (\ref{fnls_A5}) and (\ref{fnls_A20})
\begin{eqnarray}
Z^{n,2}&=&Z^{n,1}+ikb_{2}\left(-(-\partial_{xx})^{s}\frac{1}{2}\left(Z^{n,1}+Z^{n,2}\right)\right)\nonumber\\
&&+\frac{ikb_{1}}{8}P_{N}\left(|Z^{n,1}+Z^{n,2}|^{2}(Z^{n,1}+Z^{n,2})\right),\nonumber\\
&=&u^{N}+ikb_{1}\left(-(-\partial_{xx})^{s}\frac{1}{2}\left(2u^{N}+k(2b_{1}+b_{2})u_{t}^{N}\right.\right.\nonumber\\
&&\left.\left.+\frac{k^{2}(b_{1}^{2}+(b_{1}+b_{2})^{2})}{2}u_{tt}^{N}+\frac{k^{3}(b_{1}^{3}+(b_{1}+b_{2})^{3})}{6}u_{ttt}^{N}\right.\right.\nonumber\\
&&\left.\left.+\rho_{4}+k^{3}(A_{1}+B_{1})+k^{4}(A_{2}+B_{2})+e^{n,1}+e^{n,2}\right)\right)\nonumber\\
&&+\frac{ikb_{1}}{8}P_{N}\left(|G_{2}|^{2}G_{2}+\mathcal{B}(e^{n,1},e^{n,2})\right),\label{fnls_A21}
\end{eqnarray}
where
\begin{eqnarray}
||\rho_{4}||_{j}&\lesssim &k^{4}\max_{t}||\partial_{t}^{4}u^{N}||_{j},\quad j\geq 0,\label{fnls_A21b}\\
G_{2}&=&u^{N}(\tau_{n}^{1})+u^{N}(\tau_{n}^{2})+k^{3}(A_{1}+B_{1})+k^{4}(A_{2}+B_{2}),\label{fnls_A21c}\\
\mathcal{B}(e^{n,1},e^{n,2})&=&|G_{2}|^{2}(e^{n,1}+e^{n,2})+2{\rm Re}(G_{1}(\overline{e^{n,1}}+\overline{e^{n,2}})\nonumber\\
&&+|e^{n,1}+e^{n,2}|^{2}.\label{fnls_A21d}
\end{eqnarray}
As before,  we equate the two expansions (\ref{fnls_A20}) and (\ref{fnls_A21}) of $Z^{n,2}$ and, taking into account (\ref{fnls_A20b}) and (\ref{fnls_A21b})-(\ref{fnls_A21d}), we identify the coefficients in the same power of $k$ on both sides, in order to derive $B_{1}, B_{2}$ and the estimate (\ref{fnls_A2a}) for $i=2$.

We use (\ref{fnls_A21c}) and expansions of $u^{N}(\tau_{n}^{1}), u^{N}(\tau_{n}^{2})$ about $t_{n}$ to write
\begin{eqnarray}
|G_{2}|^{2}G_{2}&=&\sum_{j=0}^{4}\beta_{j}k^{j}+\rho_{5},\nonumber\\
||\rho_{5}||_{j}&\lesssim &k^{5}\max_{t}||\partial_{t}^{5}u^{N}||_{j},\quad j\geq 0,\label{fnls_A21f}
\end{eqnarray}

With the same arguments as those of the previous steps, and after tedious but straightforward calculations, one can verify that the $O(k^{j})$ terms, $j=0,1,2$, are identities. 
Equating the $O(k^{3})$ terms leads to
\begin{eqnarray}
\frac{k^{3}(b_{1}+b_{2})^{3}}{6}u_{ttt}^{N}+k^{3}B_{1}&=&\frac{k^{3}b_{1}^{3}}{6}u_{ttt}^{N}+k^{3}A_{1}\nonumber\\
&&
-i\frac{k^{3}b_{2}(b_{1}^{2}+(b_{1}+b_{2})^{2})}{4}(-\partial_{xx})^{s}u_{tt}^{N}\nonumber\\
&&+i\frac{kb_{2}}{8}P_{N}(\beta_{2}),\label{fnls_A28}
\end{eqnarray}
and equation (\ref{fnls_A28}) determines the coefficient $B_{1}$ in the expansion (\ref{fnls_A2}). 
Equating the $O(k^{4})$ terms leads to
\begin{eqnarray}
\frac{k^{4}(b_{1}+b_{2})^{4}}{24}\partial_{t}^{4}u^{N}+k^{4}B_{2}&=&\frac{k^{4}b_{1}^{4}}{24}\partial_{t}^{4}u^{N}+k^{4}A_{2}\nonumber\\
&&
-i\frac{k^{4}b_{2}(b_{1}^{3}+(b_{1}+b_{2})^{3})}{12}(-\partial_{xx})^{s}u_{ttt}^{N}\nonumber\\
&&-i\frac{k^{4}b_{2}}{2}(-\partial_{xx})^{s}(A_{1}+B_{1})\nonumber\\
&&+i\frac{kb_{2}}{8}P_{N}(\beta_{3}),\label{fnls_A24}
\end{eqnarray}
from which $B_{2}$ is derived.
Equating the $O(k^{5})$ and higher-order terms, the system for $e^{n,2}$ is now
\begin{eqnarray}
e^{n,2}+\frac{ikb_{2}}{2}(-\partial_{xx})^{s}e^{n,2}&=&\Gamma_{2}+e^{n,1}+\frac{ikb_{2}}{2}(-\partial_{xx})^{s}e^{n,1}\nonumber\\
&&+
\frac{ikb_{1}}{8}P_{N}\mathcal{B}(e^{n,1},e^{n,2}),\label{fnls_A25}
\end{eqnarray}
where $\mathcal{B}(e^{n,1},e^{n,2})$ is given by (\ref{fnls_A21d}) and
\begin{eqnarray}
\Gamma_{2}&=&\rho_{1}-\rho_{3}-\frac{ikb_{2}}{2}\left((-\partial_{xx})^{s}\left(\rho_{4}+{k^{4}}(A_{2}+B_{2})\right)\right)\nonumber\\
&&+\frac{ikb_{2}}{8}P_{N}(\beta_{4}+\rho_{5}),\label{fnls_A25b}
\end{eqnarray}
We take the real part of the inner product of (\ref{fnls_A25}) with $e^{n,2}$ to have
\begin{eqnarray}
||e^{n,2}||^{2}+\langle \frac{ikb_{2}}{2}(-\partial_{xx})^{s}e^{n,2},e^{n,2}\rangle&=&\langle \Gamma_{2},e^{n,2}\rangle+\langle e^{n,1},e^{n,2}\rangle\nonumber\\
&&+\langle \frac{ikb_{2}}{2}(-\partial_{xx})^{s}e^{n,1},e^{n,2}\rangle\nonumber\\
&&+\langle \frac{ikb_{1}}{8}P_{N}\mathcal{B}(e^{n,1},e^{n,2}),e^{n,2}\rangle.\label{fnls_A26}
\end{eqnarray}
First note that
\begin{eqnarray}
\langle\frac{ik}{2}(-\partial_{xx})^{s}e^{n,2},e^{n,2}\rangle=0.\label{fnls_A26a}
\end{eqnarray}
Using Proposition \ref{propos22}, (\ref{fnls_A5a}), (\ref{fnls_A20b}), (\ref{fnls_A21b}), and (\ref{fnls_A21f}), we have, for $\mu$ large enough,
\begin{eqnarray}
\langle\Gamma_{2},e^{n,2}\rangle\lesssim k^{5}||e^{n,2}||.\label{fnls_A26aa}
\end{eqnarray}
We use (\ref{fnls_A2}) and (\ref{fnls_A21c}) to write (\ref{fnls_A21d}) in the form
\begin{eqnarray}
\mathcal{B}(e^{n,1},e^{n,2})&=&|G_{2}|^{2}(e^{n,1}+e^{n,2})+{\rm Re}((Z^{n,1}+Z^{n,2})\overline{e^{n,1}+e^{n,2}})\nonumber\\
&&+{\rm Re}(G_{1}\overline{e^{n,1}+e^{n,2}}).\label{fnls_A26b}
\end{eqnarray}
Then, using Proposition \ref{propos22}, Lemma \ref{lemma33a}, and the estimate (\ref{fnls_A2a}) for $i=1$, for $\mu$ large enough and $k$ sufficiently small, it holds that
\begin{eqnarray}
\langle\mathcal{B}(e^{n,1},e^{n,2}),e^{n,2}\rangle\lesssim k^{5}||e^{n,2}||+||e^{n,2}||^{2}.\label{fnls_A26c}
\end{eqnarray}
In addition, the estimate (\ref{fnls_A2a}) for $i=1$ and (\ref{fnls_A18}) imply that
\begin{eqnarray}
\langle e^{n,1},e^{n,2}\rangle&\lesssim & k^{5}||e^{n,2}||,\label{fnls_A26d}\\
\langle \frac{ikb_{2}}{2}(-\partial_{xx})^{s}e^{n,1},e^{n,2}\rangle&\lesssim & k^{5}||e^{n,2}||,\label{fnls_A26e}
\end{eqnarray}

The application of (\ref{fnls_A26a})-(\ref{fnls_A26e}) to (\ref{fnls_A26}) yields, for $\mu$ large and $k$ small enough
\begin{eqnarray*}
||e^{n,2}||^{2}\lesssim k^{5}||e^{n,2}||+k||e^{n,2}||^{2},
\end{eqnarray*}
from which  (\ref{fnls_A2a}) follows in the case $j=2$, for $k$ small enough. Note, in addition, that this implies, using  (\ref{fnls_A26b}) and Proposition \ref{propos22} in (\ref{fnls_A25b}), that
\begin{eqnarray}
k||(-\partial_{xx})^{s}e^{n,2}||\lesssim k^{5}.\label{fnls_A29}
\end{eqnarray}

\medskip
The last step of the proof consists of verifying (\ref{fnls_A3}), (\ref{fnls_A3a}). 
We expand (\ref{fnls_A3}) in the form
\begin{eqnarray}
Z^{n,3}&=&u^{N}+ku_{t}^{N}+\frac{k^{2}}{2}u_{tt}^{N}+\frac{k^{3}}{6}u_{ttt}^{N}\nonumber\\
&&+\frac{k^{4}}{24}\partial_{t}^{4}u^{N}+\rho_{6}+e^{n,3},\label{fnls_A31}
\end{eqnarray}
where $\rho_{6}\in S_{N}$ satisfying
\begin{eqnarray}
||\rho_{6}||_{j}\lesssim k^{5}\max_{t}||\partial_{t}^{5}u^{N}||_{j},\quad j\geq 0.\label{fnls_A32}
\end{eqnarray}
From (\ref{fnls3_6}) with $i=3$, we have, using (\ref{fnls_A20}) and (\ref{fnls_A31})
\begin{eqnarray}
Z^{n,3}&=&Z^{n,2}+ikb_{3}\left(-(-\partial_{xx})^{s}\frac{1}{2}\left(Z^{n,2}+Z^{n,3}\right)\right)\nonumber\\
&&+\frac{ikb_{1}}{8}P_{N}\left(|Z^{n,2}+Z^{n,3}|^{2}(Z^{n,2}+Z^{n,3})\right),\nonumber\\
&=&u^{N}+ikb_{3}\left(-(-\partial_{xx})^{s}\frac{1}{2}\left(2u^{N}+k(b_{1}+b_{2}+1)u_{t}^{N}\right.\right.\nonumber\\
&&\left.\left.+\frac{k^{2}(1+(b_{1}+b_{2})^{2})}{2}u_{tt}^{N}+\frac{k^{3}(1+(b_{1}+b_{2})^{3})}{6}u_{ttt}^{N}\right.\right.\nonumber\\
&&\left.\left.
+\rho_{7}+k^{3}B_{1}+k^{4}B_{2}+e^{n,2}+e^{n,3}\right)\right)\nonumber\\
&&+\frac{ikb_{1}}{8}P_{N}\left(|G_{3}|^{2}G_{3}+\mathcal{C}(e^{n,2},e^{n,3})\right),\label{fnls_A33}
\end{eqnarray}
where
\begin{eqnarray}
||\rho_{7}||_{j}&\lesssim &k^{4}\max_{t}||\partial_{t}^{4}u^{N}||_{j},\quad j\geq 0,\label{fnls_A33b}\\
G_{3}&=&u^{N}(\tau_{n}^{2})+u^{N}(t_{n+1})+k^{3}B_{1}+k^{4}B_{2},\label{fnls_A33c}\\
\mathcal{C}(e^{n,2},e^{n,3})&=&|G_{3}|^{3}(e^{n,2}+e^{n,3})+2{\rm Re}(G_{3}(\overline{e^{n,2}}+\overline{e^{n,3}})\nonumber\\
&&+|e^{n,2}+e^{n,3}|^{2}.\label{fnls_A33d}
\end{eqnarray}
Again, we equate the two expansions (\ref{fnls_A31}) and (\ref{fnls_A33}) of $Z^{n,3}$ and, taking into account (\ref{fnls_A32})-(\ref{fnls_A33d}), we identify the coefficients in the same power of $k$ on both sides, in order to derive (\ref{fnls_A3}), (\ref{fnls_A3a}).

In this case, we use (\ref{fnls_A33c}) and expansions of $u^{N}(t_{n+1}), u^{N}(\tau_{n}^{2})$ about $t_{n}$ to get the expansion
\begin{eqnarray}
|G_{3}|^{2}G_{3}&=&\sum_{j=0}^{4}\gamma_{j}k^{j}+\rho_{8},\nonumber\\
||\rho_{8}||_{j}&\lesssim &k^{5}\max_{t}||\partial_{t}^{5}u^{N}||_{j},\quad j\geq 0,\label{fnls_A33f}
\end{eqnarray}
As before, careful calculations imply that each $O(k^{j})$ term, $j=0,1,2$, is the same on both sides. Then using differentiation of (\ref{fnls2_3}) with respect to $t$ up to three times, the properties $b_{1}+b_{2}+b_{3}=1, b_{1}=b_{3}$, (\ref{fnls_A33f}), the equations (\ref{fnls_A28}) and (\ref{fnls_A24}) for $B_{1}$ and $B_{2}$ respectively, and after some algebra, one can obtain the identity for the $O(k^{3})$ and for the $O(k^{4})$ terms. 

Equating the $O(k^{5})$ and higher-order terms, the system for $e^{n,3}$ takes now the form
\begin{eqnarray}
e^{n,3}+\frac{ikb_{3}}{2}(-\partial_{xx})^{s}e^{n,3}&=&\Gamma_{3}+e^{n,2}+\frac{ikb_{3}}{2}(-\partial_{xx})^{s}e^{n,2}\nonumber\\
&&+
\frac{ikb_{1}}{8}P_{N}\mathcal{C}(e^{n,2},e^{n,3}),\label{fnls_A36}
\end{eqnarray}
where $\mathcal{C}(e^{n,2},e^{n,3})$ is given by (\ref{fnls_A33d}) and
\begin{eqnarray*}
\Gamma_{3}&=&\rho_{3}-\rho_{6}-\frac{ikb_{3}}{2}\left((-\partial_{xx})^{s}\left(\rho_{7}+{k^{4}}B_{2}\right)\right)\nonumber\\
&&+\frac{ikb_{3}}{8}P_{N}(\gamma_{4}+\rho_{8}),\label{fnls_A36b}
\end{eqnarray*}
We take the real part of the inner product of (\ref{fnls_A36}) with $e^{n,3}$ to have
\begin{eqnarray}
||e^{n,3}||^{2}+\langle \frac{ikb_{3}}{2}(-\partial_{xx})^{s}e^{n,3},e^{n,3}\rangle&=&\langle \Gamma_{3},e^{n,3}\rangle+\langle e^{n,2},e^{n,3}\rangle\nonumber\\
&&+\langle \frac{ikb_{2}}{2}(-\partial_{xx})^{s}e^{n,2},e^{n,3}\rangle\nonumber\\
&&+\langle \frac{ikb_{1}}{8}P_{N}\mathcal{C}(e^{n,2},e^{n,3}),e^{n,3}\rangle.\label{fnls_A37}
\end{eqnarray}
First note that, as in the previous cases
\begin{eqnarray}
\langle\frac{ik}{2}(-\partial_{xx})^{s}e^{n,3},e^{n,3}\rangle=0.\label{fnls_A37a}
\end{eqnarray}
Using Proposition \ref{propos22}, (\ref{fnls_A20b}), (\ref{fnls_A32}), (\ref{fnls_A33b}), and (\ref{fnls_A33f}), we have, for $\mu$ large enough,
\begin{eqnarray}
\langle\Gamma_{3},e^{n,3}\rangle\lesssim k^{5}||e^{n,3}||.\label{fnls_A38a}
\end{eqnarray}
We use (\ref{fnls_A2}) and (\ref{fnls_A33c}) to write (\ref{fnls_A33d}) in the form
\begin{eqnarray}
\mathcal{C}(e^{n,21},e^{n,3})&=&|G_{3}|^{2}(e^{n,2}+e^{n,3})+{\rm Re}((Z^{n,2}+Z^{n,32})\overline{e^{n,2}+e^{n,3}})\nonumber\\
&&+{\rm Re}(G_{3}\overline{e^{n,2}+e^{n,3}}).\label{fnls_A38b}
\end{eqnarray}
Then, using Proposition \ref{propos22}, Lemma \ref{lemma33a}, and the estimate (\ref{fnls_A2a}) for $i=2$, for $\mu$ large enough and $k$ sufficiently small, it holds that
\begin{eqnarray}
\langle\mathcal{C}(e^{n,21},e^{n,3}),e^{n,3}\rangle\lesssim k^{5}||e^{n,3}||+||e^{n,3}||^{2}.\label{fnls_A38c}
\end{eqnarray}
In addition, the estimate (\ref{fnls_A2a}) for $i=2$ and (\ref{fnls_A29}) imply that
\begin{eqnarray}
\langle e^{n,2},e^{n,3}\rangle&\lesssim & k^{5}||e^{n,3}||,\label{fnls_A38d}\\
\langle \frac{ikb_{2}}{2}(-\partial_{xx})^{s}e^{n,2},e^{n,3}\rangle&\lesssim & k^{5}||e^{n,3}||,\label{fnls_A38e}
\end{eqnarray}

The application of (\ref{fnls_A37a})-(\ref{fnls_A38e}) to (\ref{fnls_A37}) yields, for $\mu$ large and $k$ small enough
\begin{eqnarray*}
||e^{n,3}||^{2}\lesssim k^{5}||e^{n,3}||+k||e^{n,3}||^{2},
\end{eqnarray*}
from which  (\ref{fnls_A3}), (\ref{fnls_A3a}) follow, and the proof is complete.
\end{proof}
\subsection{General $f$}
\label{secA3}
As observed in the proof of Propositions \ref{proposA0} and \ref{proposA1}, one can identify key steps that may be generalized for methods (\ref{44}) of arbitrary stages $q=3^{p-1}, p\geq 1$. They can be summarized as follows:
\begin{itemize}
\item The aim of the proof is to verify the expansions
\begin{eqnarray*}
Z^{n,j}=u^{N}(\tau_{n}^{j})+k^{p+1}A_{1}^{(i)}+\cdots+k^{2p}A_{p}^{(i)}+e^{n,j},
\end{eqnarray*}
where for $1\leq j\leq q-1$,
$$\tau_{n}^{j}=t_{n}+k\sum_{l=1}^{j}b_{l},$$ $A_{1}^{(j)},\ldots,A_{p}^{(j)}, e^{n,j}\in S_{N}$, with $||e^{n,j}||\lesssim k^{2p+1}$, and at the final step
$$Z^{n,s}=u^{N}(t_{n+1})+e^{n,s},\quad ||e^{n,s}||\lesssim k^{2p+1}.$$
\item On each stage $j=2,\ldots,q$, we expand the corresponding formula in (\ref{fnls3_6}) for $Z^{n,j}$ and compare with Taylor's expansion about $t_{n}$, equating both expansions in powers of $k$. From the corresponding equations for $k^{p+1},\ldots,k^{2p}$, the coefficients  $A_{1}^{(j)},\ldots,A_{p}^{(j)}$ are derived, while the equation for $k^{2p+1}$ and higher-order terms will be used to estimate the remainders $e^{n,j}$, from the estimates obtained in the previous step for $e^{n,j-1}$ and $(-\partial_{xx})^{s}e^{n,j-1}$. (In the case $j=1$, this role is played by $u^{N}$ and Proposition \ref{propos22} is used.)
\item The specific form of $f$ is mainly used in the derivation of the equations for each power of $k$. For the cubic case $f(u)=|u|^{2}u$, they come from the expansions at $t_{n}$ of $|u|$ (or $|u|^{2}$). The nonlinear terms in the expansions must be separated into two main contributions, one of them independent of the $e^{n,j}$.
\item The estimates for $e^{n,j}$, from taking the (real part of the) inner product of the corresponding equation and $e^{n,j}$, requires the control of the resulting elements using:
\begin{enumerate}
\item The estimates obtained in the previous step for $e^{n,j-1}$ and $(-\partial_{xx})^{s}e^{n,j-1}$.
\item The use of Proposition \ref{propos22}, whih depends on the regularity of the exact solution $u$ and $f$.
\item Writing the nonlinear term involving $e^{n,j}$ and $e^{n,j-1}$ in a suitable form  to apply Lemma \ref{lemma33a}, and avoiding, in the final inequality for $||e^{n,j}||$, the presence of powers of $||e^{n,j}||$ of order higher than two.
\end{enumerate}
\end{itemize}
\end{document}